\documentclass[11pt]{article}

\usepackage[margin=1in]{geometry}
\usepackage{amsmath,amssymb,amsthm}
\usepackage{bm}
\usepackage{graphicx}
\usepackage{tikz}
\usetikzlibrary{arrows.meta,calc}

\usepackage[sort&compress]{natbib}
\bibpunct[, ]{[}{]}{,}{n}{}{,}%

\usepackage[nameinlink,noabbrev]{cleveref}

\providecommand{\texorpdfstring}[2]{#1}

\theoremstyle{plain}
\newtheorem{theorem}{Theorem}
\newtheorem{lemma}{Lemma}
\newtheorem{proposition}{Proposition}
\newtheorem{corollary}{Corollary}
\newtheorem{assumption}{Assumption}
\theoremstyle{definition}

\newtheorem{example}{Example}
\theoremstyle{remark}
\newtheorem{remark}{Remark}

\makeatletter
\renewenvironment{proof}[1]
  {\par\addvspace{6pt plus 2pt}\noindent\hspace*{1em}{\itshape #1\enskip}\ignorespaces}
  {\par\addvspace{6pt plus 2pt}}
\makeatother
\newcommand{\Halmos}{\qed}

\DeclareMathOperator*{\argmin}{arg\,min}

\newcommand{\up}{\rule{0pt}{13pt}}
\newcommand{\down}{\rule[-5pt]{0pt}{0pt}}

\newcommand{\FIGURE}[3]{%
  \centering
  #1\par
  \caption{#2}%
  \if\relax\detokenize{#3}\relax\else
    \par\smallskip\footnotesize\textit{Note.} #3\par
  \fi}
\newcommand{\TABLE}[3]{%
  \centering
  \caption{#1}%
  \smallskip
  {\small #2}\par
  \if\relax\detokenize{#3}\relax\else
    \par\smallskip\footnotesize\textit{Note.} #3\par
  \fi}

\crefname{theorem}{Theorem}{Theorems}
\Crefname{theorem}{Theorem}{Theorems}
\crefname{lemma}{Lemma}{Lemmas}
\Crefname{lemma}{Lemma}{Lemmas}
\crefname{proposition}{Proposition}{Propositions}
\Crefname{proposition}{Proposition}{Propositions}
\crefname{corollary}{Corollary}{Corollaries}
\Crefname{corollary}{Corollary}{Corollaries}
\crefname{assumption}{Assumption}{Assumptions}
\Crefname{assumption}{Assumption}{Assumptions}
\crefname{definition}{Definition}{Definitions}
\Crefname{definition}{Definition}{Definitions}
\crefname{example}{Example}{Examples}
\Crefname{example}{Example}{Examples}
\crefname{remark}{Remark}{Remarks}
\Crefname{remark}{Remark}{Remarks}
\crefname{figure}{Figure}{Figures}
\Crefname{figure}{Figure}{Figures}
\crefname{table}{Table}{Tables}
\Crefname{table}{Table}{Tables}
\crefname{section}{Section}{Sections}
\Crefname{section}{Section}{Sections}

\newcommand{\R}{\mathbb R}

\newcommand{\bmzero}{\bm 0}

\newcommand{\UE}{\mathrm{UE}}
\newcommand{\SO}{\mathrm{SO}}
\newcommand{\PoA}{\mathrm{PoA}}
\newcommand{\Cost}{\mathrm{Cost}}
\newcommand{\Nondiff}{\mathrm{Nondiff}}

\newcommand{\intO}{\operatorname{int}}

\newcommand{\supp}{\operatorname{supp}}

\title{Regularity of the Price of Anarchy in Routing Games over the Multivariate Demand Space}
\author{Jiawen Wang\textsuperscript{a}
\and
Hu Shao\textsuperscript{b}
\and
Ting Wu\textsuperscript{a,*}}
\date{%
  \small \textsuperscript{a}\,School of Mathematics, Nanjing University, Nanjing, Jiangsu 210093, China\\[3pt]
  \small \textsuperscript{b}\,School of Mathematics, China University of Mining and Technology, Xuzhou, Jiangsu 221116, China\\[3pt]
  \small \textsuperscript{*}\,Corresponding author. Email: \texttt{tingwu@nju.edu.cn}}

\begin{document}

\maketitle

\begin{abstract}
We study the price of anarchy (PoA) in nonatomic routing games as a function
of the full vector of origin--destination demands. We show that the Beckmann
value function is continuously differentiable on the positive demand cone,
with gradient equal to the equilibrium origin--destination cost vector; hence
the PoA is continuously differentiable on the regular demand region. We prove
that the nondifferentiability set has Lebesgue measure zero and, when every
edge cost is strictly increasing or constant, is covered by countably many
codimension-one submanifolds. Under the same cost structure, all equilibrium
objects are cellwise as smooth as the edge costs. For affine costs, the
exceptional sets become finite unions of lower-dimensional polyhedral cells
linked by an exact scaling law, and the maximum of the PoA over a cell
closure is attained on the cell boundary, generalizing the one-dimensional
endpoint result. The theory is illustrated on the Sioux Falls benchmark
network.
\end{abstract}


\bigskip

\section{Introduction}

Nonatomic routing games provide the standard mathematical model of selfish
traffic distribution over networks. In this model, a continuum of users
travels between origin--destination (OD) pairs of a directed multigraph.
The cost of each edge is a nondecreasing function of its load.
A Wardrop
equilibrium is a flow in which, for every OD pair, only minimum-cost paths
are used (Wardrop~\cite{wardrop1952}). Wardrop equilibria admit a classical
characterization as the optimal solutions of a convex optimization program
(Beckmann, McGuire, and Winsten~\cite{beckmann1956}), and their inefficiency is usually quantified by the
\emph{price of anarchy} (PoA) (Koutsoupias and Papadimitriou~\cite{koutsoupias1999}; Papadimitriou~\cite{papadimitriou2001}), the
ratio between the total equilibrium cost and the minimum social cost. These
games are a central instance of the broader class of congestion games
(Rosenthal~\cite{rosenthal1973}; Nisan et al.~\cite{nisan2007}).

The classical theory of the PoA has focused on worst-case bounds over broad
families of networks: for nonatomic congestion games with affine costs the
PoA is at most $4/3$ and this bound is sharp (Roughgarden and
Tardos~\cite{roughgarden_tardos2002}),
and sharp bounds and growth rates are known for polynomial costs and beyond
(Roughgarden~\cite{roughgarden2003}; Dumrauf and
Gairing~\cite{dumrauf_gairing2006}; Roughgarden and
Tardos~\cite{roughgarden_tardos2004}; Chau and Sim~\cite{chau_sim2003};
Perakis~\cite{perakis2007}).
Empirical studies, however, show that in real networks the PoA behaves quite
differently from what worst-case analysis suggests: as the demand grows, the
PoA oscillates with a number of nonsmooth spikes, stays far below the
worst-case bounds, and returns to one in the heavily congested regime
(Youn et al.~\cite{youn2008}; Monnot et al.~\cite{monnot2017}; O'Hare et
al.~\cite{ohare2016}).

These observations have motivated a research program that studies the PoA
\emph{as a function of the traffic demand}. On the one hand, equilibria are
asymptotically efficient: under general assumptions the PoA converges to one
in light and in heavy traffic (Colini-Baldeschi et
al.~\cite{colini2019,colini2020}), and the
scalability framework of \citeauthor{wu2021}~\cite{wu2021} establishes this convergence along
arbitrary demand growth patterns; analogous convergence results have been
obtained along sequences of atomic congestion games (Wu et
al.~\cite{wu_mrx2023}; Cominetti et al.~\cite{cominetti2022approximation}). On the
other hand, the intermediate demand regime---where the PoA oscillates and
exhibits sharp kinks associated with changes of the set of optimal
routes---remains the object of active investigation (O'Hare et
al.~\cite{ohare2016}; Englert et al.~\cite{englert2010}).
Understanding the fine behavior of the PoA in this regime requires, first of
all, a precise description of its regularity as a function of the demand.

The dependence of equilibria and of the PoA on the demand has been studied
along three complementary lines. First, a one-dimensional regularity theory:
for single-OD networks, \citeauthor{cds_mp2024}~\cite{cds_mp2024} give a complete description in
which the PoA is $C^1$ at regular demands, its nonsmoothness is concentrated
at breakpoints where the set of optimal paths changes, and for affine costs
its maximum is attained at a breakpoint; for multi-commodity networks,
\citeauthor{cds_pt2024}~\cite{cds_pt2024} established smoothness along one-parameter demand curves
and resolved the phase-transition conjecture of \citeauthor{ohare2016}~\cite{ohare2016} under
proportional demand growth, while \citeauthor{cds_mo2024}~\cite{cds_mo2024} studied the monotonicity
of equilibria on the full demand space. Second, a metric sensitivity theory:
\citeauthor{wu_mohring2023}~\cite{wu_mohring2023} endowed the space of games sharing a network with a
metric and proved that the PoA is continuous everywhere and pointwise
H\"older continuous under simultaneous perturbations of demands and cost
functions, with the H\"older exponent improving to one precisely when the
cost functions are constant or strictly increasing. In the same work, the
differentiability analysis of the PoA was explicitly identified as an open
direction: observing that the metric structure of the full game space is too
complicated to allow a direct attack, the authors suggested that parametric
demands in a finite-dimensional Euclidean space may pave a way for future
work on the differentiability of the PoA. Third, a classical sensitivity
theory of equilibria themselves, addressing monotonicity, directional
differentiability, and parametric computation
(Hall~\cite{hall1978}; Dafermos and Nagurney~\cite{dafermos_nagurney1984};
Dafermos~\cite{dafermos1980}; Patriksson~\cite{patriksson2004}; Josefsson and
Patriksson~\cite{josefsson_patriksson2007}; Lu~\cite{lu2008}; Englert et
al.~\cite{englert2010}; Takalloo and Kwon~\cite{takalloo_kwon2020}; Klimm and
Warode~\cite{klimm_warode2019,klimm_warode2022}; Harks and Klimm~\cite{harks_klimm2015}).

Beyond the demand-dependence question itself, the literature has developed a
rich array of model extensions. These include congestion games with atomic
(weighted or splittable) players (Awerbuch et al.~\cite{awerbuch2013}),
stochastic and atomic variants (Cominetti et
al.~\cite{cominetti2019stochastic}; Wu et al.~\cite{wu_mrx2023}), dynamic
flows over time (Correa, Cristi, and Oosterwijk~\cite{correa_cristi_oosterwijk2021}; Graf and Harks~\cite{graf_harks2023}) and multi-modal dynamic traffic assignment (Pi, Ma, and Qian~\cite{pi2019multimodal}; Huang et al.~\cite{huang2019dynamic}),
capacity constraints and restricted route sets
(Correa et al.~\cite{correa2004}; Jahn et al.~\cite{jahn2005}), elastic demands
(Chau and Sim~\cite{chau_sim2003}), and boundedly rational route-choice behavior
(Li et al.~\cite{li_wang_nie2024}). Network-structure restrictions form another active
direction, with recent progress on series-parallel and matroid congestion
games
(Fujishige et al.~\cite{fujishige2017}; Hao and
Michini~\cite{hao_michini2024mp,hao_michini2024sagt}; Fokkema et
al.~\cite{fokkema_hoeksma_uetz2024}).
A parallel stream studies intervention mechanisms that reduce the
inefficiency of equilibria, such as tolls and coordination mechanisms
(Colini-Baldeschi et al.~\cite{colini_tolls2018}; Christodoulou et
al.~\cite{christodoulou2014mechanisms}), and congestion pricing for
multimodal urban networks (Zheng, Rérat, and
Geroliminis~\cite{zheng2015pricing}; Liu, Ma, and
Qian~\cite{liu2022curbside}),
and new analytical techniques continue to sharpen the available bounds across
model classes (Kashaev~\cite{kashaev2025}). Finally, a structural literature has
clarified the paradoxes that equilibria may exhibit, as well as the topology
underlying their uniqueness
(Braess~\cite{braess1968}; Dimos et al.~\cite{dimos_fotakis2023}; Matsubayashi
and Saito~\cite{matsubayashi_saito2023};
Milchtaich~\cite{milchtaich2000,milchtaich2005}; Meunier and
Pradeau~\cite{meunier_pradeau2014}).

Yet, between the one-dimensional smoothness results and the metric H\"older
bounds, the \emph{geometry} of the PoA's regularity on the full demand cone
remains open. We know that the PoA varies continuously, even pointwise
H\"older-continuously, under perturbations (Wu and
M\"ohring~\cite{wu_mohring2023}), and we know
where it fails to be smooth along one-dimensional slices of the demand space
(Cominetti et al.~\cite{cds_mp2024,cds_pt2024}); but we do not know where it is smooth as a
multivariate map, what its singular set looks like, or how much smoothness
structural assumptions provide. The non-proportional counterexamples of
\citeauthor{cds_pt2024}~\cite{cds_pt2024} indicate that genuinely multivariate phenomena are to be
expected. The present paper develops a regularity theory that answers these
questions on the full $n$-dimensional demand cone. In particular, it carries
out, on the demand space, precisely the differentiability analysis of the PoA
that \citeauthor{wu_mohring2023}~\cite{wu_mohring2023} left for future work, and it answers their open
question on when the PoA is differentiable.

Our analysis is built on a multivariate value-function framework. We show
that the Beckmann value function is convex and continuously differentiable
on the whole positive demand cone, and that its gradient is precisely the
equilibrium origin--destination cost vector (\cref{lem:value-C1}). Building
on this, we introduce the active-set regions---the demand regions induced by
a fixed set of positive-flow paths---and prove that the user-equilibrium cost
is locally $C^1$ on the interior of every active-set region
(\cref{thm:local-C1-CUE}), so that the PoA is $C^1$ on the regular region
(\cref{C1_regularity_of_PoA}). Notably, the analysis on the user-equilibrium
side requires only continuous nondecreasing costs; in particular, strict
monotonicity of the edge costs is not required on the UE side.

We then characterize the nondifferentiability set. Nondifferentiability can
occur only in the exceptional region where active sets switch; but that
candidate region can be misleading, as we show through an example, based on a
fat Cantor set, in which the exceptional region has positive measure and yet
the equilibrium cost---and hence the PoA---is infinitely smooth
(\cref{ex:E_positive_measure}). The actual nondifferentiability set is much
thinner: it always has Lebesgue measure zero
(\cref{prop:nondiff-measure-zero}) and, under a standard cost structure in
which every edge cost is either strictly increasing or constant, it is
covered by at most countably many codimension-one $C^1$ submanifolds of the
demand space (\cref{thm:nondiff-countable-manifold-cover}). In this precise
sense, the breakpoints of the one-dimensional theory are the traces of these
submanifolds along demand curves.

Higher-order regularity is not free: even with smooth costs, the equilibrium
cost may fail to be $C^3$ in the interior of an active-set region, as we show
in a three-commodity example (\cref{ex:Cr-degenerate-example}). Under the
strictly-increasing/constant cost structure, however, the equilibrium
edge-cost vector, the OD-cost vector, and the user-equilibrium cost are
cellwise $C^r$ (\cref{thm:Cr_regularity_UEcost}), and so is the PoA on the
intersection of the UE and SO regular regions (\cref{cor:cr-poa}). For affine
costs, the exceptional sets become finite unions of lower-dimensional
polyhedral cells, linked
by the exact scaling law $\mathcal E^{\mathrm{UE}} = 2\,\mathcal
E^{\mathrm{SO}}$ (\cref{prop:affine-exceptional-polyhedral}), and the PoA is
a quotient of quadratic functions on each cell. In this setting we also
establish a boundary maximum principle (\cref{thm:poa-cell-boundary-maximum}):
the maximum of the PoA over a cell closure is always attained on the cell
boundary, generalizing the single-OD endpoint maximum principle of
\citeauthor{cds_mp2024}~\cite{cds_mp2024} to the multivariate demand space; in a
three-commodity example, the maximum is attained in the relative interior of
a one-dimensional boundary edge (\cref{ex:three-commodity-affine}).

Finally, the theory is illustrated numerically on the Sioux Falls benchmark
network (\cref{sec:numerical}). On demand slices spanned by three
representative OD pairs of the $528$-dimensional demand cone, the
nondifferentiability set of the PoA displays precisely the predicted
geometry---curved codimension-one sheets under BPR costs, planar polyhedral
cells under affine costs; the gradient identity $D_\mu V=\lambda$ underlying
the whole framework is verified to within $10^{-9}$; the scaling law
$\mathcal E^{\mathrm{UE}}=2\,\mathcal E^{\mathrm{SO}}$ is confirmed exactly
on the demand ray; and the maximum of the PoA over an active-set cell is
observed in the relative interior of a boundary edge---a network-scale
instance of the boundary maximum principle.

Technically, our analysis lifts the value-function approach from scalar to
vector demands. The subdifferential of the Beckmann value function is
identified with the set of equilibrium cost vectors through the
perturbation-duality theory of convex programs
(Rockafellar~\cite{rockafellar1970}; Rockafellar and
Wets~\cite{rockafellar_wets1998}). Regularity on active-set
interiors is extracted from restricted Beckmann problems, and the local geometry
near active-set boundaries is resolved by constructing $C^1$ extension
branches via an implicit-function argument, in which a kernel-based reduction
handles the degeneracy caused by constant-cost edges. The measure-zero
property of the nondifferentiability set follows from the almost-everywhere
second-order differentiability of the convex value function
(Alexandrov~\cite{alexandrov1939}; Mignot~\cite{mignot1976}). The passage from this local manifold
description to a global countable cover relies on Lindel\"of's covering
theorem: since the demand cone is second countable, every open cover of the
nondifferentiability set admits a countable subcover (Kelley~\cite{kelley1975}),
which yields the desired ``at most countable'' conclusion.

These results turn the mechanism picture of \citeauthor{ohare2016}~\cite{ohare2016} into a geometric
theory on the demand space, and complement the metric sensitivity bounds of
\citeauthor{wu_mohring2023}~\cite{wu_mohring2023} with a structural description of where and why the PoA
fails to be smooth. They thus provide the regularity foundation for
comparative statics of the PoA with respect to multidimensional demand.

The rest of the paper is organized as follows. \Cref{sec:model} introduces
the multivariate model, the active-set regions, and the $C^1$ regularity
framework. \Cref{sec:nondiff} studies the nondifferentiability set: its
localization, its measure, and its countable manifold cover.
\cref{sec:higher-order} investigates higher-order regularity, including the
$C^3$ counterexample and the cellwise $C^r$ theory. \Cref{sec:affine}
specializes to affine networks, establishes the boundary maximum principle
for the PoA, and illustrates it on a three-commodity example. \Cref{sec:numerical} illustrates the
theory on the Sioux Falls benchmark network. \Cref{sec:concluding} collects
concluding
remarks, including the precise relationship between the present results and
the differentiability question of \citeauthor{wu_mohring2023}~\cite{wu_mohring2023}. All proofs are
deferred to the appendices.

\section{Multivariate PoA Model and \(C^1\) Regularity}
\label{sec:model}

\subsection{Preliminaries and the PoA as a Function of the Demand Vector}

We consider a finite directed multigraph \(\mathcal G=(\mathcal V,\mathcal E)\) representing a traffic network, where \(\mathcal V\) is the vertex set, \(\mathcal E\) is the edge set, and \(m:=|\mathcal E|\). The network is endowed with \(n\) origin--destination pairs
\(\{(O_i,D_i)\}_{i=1}^n\). For the \(i\)-th OD pair, let \(\mathcal P_i\) denote its set of admissible paths. The set of all paths is denoted by \(\mathcal P:=\bigcup_{i=1}^n\mathcal P_i\), and we write \(k:=|\mathcal P|<\infty\).
If a path \(p\) traverses an edge \(e\), we write \(e\in p\). Here the symbol \(\in\) only records the incidence relation between edges and paths.

Let \(f=(f_p)_{p\in\mathcal P}\in\R^k\) denote the path-flow vector. Let \(C\in\R^{m\times k}\) be the edge--path incidence matrix, with entries \(C_{ep}:=\mathbf 1_{\{e\in p\}}\). The edge-load vector induced by \(f\) is denoted by \(x=Cf\in\R^m\). Each edge \(e\in\mathcal E\) is equipped with a cost function \(c_e:[0,+\infty)\to\R_+\), where \(c_e(x_e)\) is the travel cost on edge \(e\) under edge load \(x_e\). Throughout this subsection, every \(c_e\) is assumed to be continuous and nondecreasing. Given a path flow \(f\), the edge cost and the path cost are respectively defined by \(\phi_e(f):=c_e\bigl((Cf)_e\bigr)\) and \(c_p(f):=\sum_{e\in p}c_e\bigl((Cf)_e\bigr)\).
The total network travel cost induced by \(f\) is
\[
\Cost(f)=\sum_{p\in\mathcal P} f_p\,c_p(f)=\sum_{e\in\mathcal E}(Cf)_e\,c_e\bigl((Cf)_e\bigr).
\]

The demand vector is denoted by \(\bm\mu=(\mu_1,\dots,\mu_n)\in U^\circ\), where \(U^\circ:=(0,+\infty)^n\), and the component \(\mu_i\) represents the traffic demand of the \(i\)-th OD pair. Let \(R\in\R^{n\times k}\) be the OD--path incidence matrix, with entries \(R_{ip}:=\mathbf 1_{\{p\in\mathcal P_i\}}\). The demand constraints can then be written as \(Rf=\bm\mu\). For a given demand \(\bm\mu\in U^\circ\), the feasible path-flow set is \(\mathcal F(\bm\mu):=\{f\in\R_{\geq 0}^k:\ Rf=\bm\mu\}\).
Since each \(\mathcal P_i\) is nonempty, \(\mathcal F(\bm\mu)\) is nonempty for every \(\bm\mu\in U^\circ\). Moreover, \(\mathcal F(\bm\mu)\) is a compact convex polytope.

For a given demand \(\bm\mu\in U^\circ\), a feasible flow \(f^{*\UE}\in\mathcal F(\bm\mu)\) is called a Wardrop equilibrium flow if, for each OD pair \(i\), there exists a scalar \(\lambda_i\) such that
\begin{equation}
c_p(f^{*\UE})=\lambda_i
\qquad
\text{whenever }p\in\mathcal P_i,\ f_p^{*\UE}>0,
\label{eq:wardrop-equal}
\end{equation}
and
\begin{equation}
c_p(f^{*\UE})\ge \lambda_i
\qquad
\text{whenever }p\in\mathcal P_i,\ f_p^{*\UE}=0.
\label{eq:wardrop-ineq}
\end{equation}
Thus, for each OD pair, positive flow is assigned only to paths of minimum cost. The scalar \(\lambda_i\) is called the equilibrium cost of the \(i\)-th OD pair at this equilibrium.

A basic value-level property of Wardrop equilibria is that, for any fixed demand \(\bm\mu\in U^\circ\), the Wardrop equilibrium flow \(f^{*\UE}\) need not be unique, but the equilibrium edge-cost vector induced by any Wardrop equilibrium is unique. This uniqueness is classical; see, for instance, \citeauthor{smith1979}~\cite{smith1979} and \citeauthor{patriksson1994}~\cite[Theorem~2.1]{patriksson1994}. For completeness we recall a short argument: if two Beckmann minimizers induce edge-load vectors \(x\) and \(y\), then convexity of the Beckmann potential implies that \(x\mapsto\int_0^{x_e}c_e(t)\,dt\) is affine on the segment \([x_e,y_e]\) for every edge \(e\); hence \(c_e\) is constant on \([x_e,y_e]\) by continuity, and therefore \(c_e(x_e)=c_e(y_e)\). Hence the equilibrium edge cost \(\tau_e(\bm\mu)\) is a well-defined function of the demand vector \(\bm\mu\). Consequently, for each OD pair \(i\), its equilibrium cost is defined by \(\lambda_i(\bm\mu):=\min_{p\in\mathcal P_i}\sum_{e\in p}\tau_e(\bm\mu)\).
Accordingly, \(\bm\lambda(\bm\mu):=(\lambda_1(\bm\mu),\dots,\lambda_n(\bm\mu))^\top\) is a well-defined vector-valued function of \(\bm\mu\).

The user-equilibrium total cost is therefore defined without ambiguity by
\[
C^{\UE}(\bm\mu):=\sum_{i=1}^n\mu_i\lambda_i(\bm\mu)=\bm\mu^\top\bm\lambda(\bm\mu).
\]
This definition does not depend on the choice of a particular Wardrop equilibrium path-flow minimizer. Equivalently, for any Wardrop equilibrium flow \(f^{*\UE}\), \(C^{\UE}(\bm\mu)=\sum_{e\in\mathcal E}(Cf^{*\UE})_e\,c_e\bigl((Cf^{*\UE})_e\bigr)\).

To characterize Wardrop equilibria, introduce the Beckmann potential
\[
B(f):=\sum_{e\in\mathcal E}\int_0^{(Cf)_e}c_e(t)\,dt.
\]
Since each \(c_e\) is continuous and nondecreasing, the map \(x_e\mapsto \int_0^{x_e}c_e(t)\,dt\) is convex and differentiable. Hence \(B\) is a convex differentiable function of the path-flow vector \(f\). The following classical result gives the Beckmann characterization of Wardrop equilibria.

\begin{proposition}[Beckmann--McGuire--Winsten characterization]
	Under the standing assumptions, fix \(\bm\mu\in U^\circ\). A minimizer of the Beckmann problem \(\min_{f\in\mathcal F(\bm\mu)}B(f)\) induces a Wardrop equilibrium. Conversely, the load profile of any Wardrop equilibrium corresponds to an optimal load profile of the Beckmann problem.
\end{proposition}

The corresponding Beckmann value function is defined by
\[
V(\bm\mu):=\min_{f\in\mathcal F(\bm\mu)}B(f),
\qquad
\bm\mu\in U^\circ.
\]
Since \(\mathcal F(\bm\mu)\) is nonempty and compact and \(B\) is continuous, the above minimum is attained.

The system optimum is defined by minimizing the total network travel cost under a fixed demand vector \(\bm\mu\). Its value function is \(C^{\SO}(\bm\mu):=\min_{f\in\mathcal F(\bm\mu)}\Cost(f)\).
Since \(\mathcal F(\bm\mu)\) is nonempty and compact and \(\Cost\) is continuous, the above minimum is attained. Thus \(C^{\SO}(\bm\mu)\) is a well-defined value function. Any feasible flow attaining this minimum is called a system-optimal flow and is denoted by \(f^{*\SO}\in\operatorname*{arg\,min}_{f\in\mathcal F(\bm\mu)}\Cost(f)\).

If each edge cost satisfies \(c_e\in C^1([0,+\infty))\), define the transformed cost
\[
\bar c_e(x):=\frac{d}{dx}\bigl[xc_e(x)\bigr]=c_e(x)+x c'_e(x).
\]
When \(x\mapsto x c_e(x)\) is convex, the system-optimum problem is equivalent to a Beckmann-type problem on the same network with edge costs \(\bar c_e\).

\begin{proposition}[Beckmann reduction for the system-optimum value function]
	\label{prop:BMW_for_SOcost}
	Suppose that \(c_e\in C^1([0,+\infty))\) for every \(e\in\mathcal E\), and that \(x\mapsto x c_e(x)\) is convex. Then
	\[
	C^{\SO}(\bm\mu)=\min_{f\in\mathcal F(\bm\mu)}\sum_{e\in\mathcal E}\int_0^{(Cf)_e}\bar c_e(t)\,dt,
	\]
	where \(\bar c_e(x)=c_e(x)+x c'_e(x)\). Equivalently, the system-optimum value function coincides with the Beckmann value function associated with the transformed edge costs \((\bar c_e)_{e\in\mathcal E}\).
\end{proposition}

For the subsequent analysis, we impose the following assumption.
\begin{assumption}[Positivity of the SO cost]
	\label{ass:SO_Cost_positive}
	On the positive demand domain \(U^\circ=(0,+\infty)^n\), assume that \(C^{\SO}(\bm\mu)>0\) for every \(\bm\mu\in U^\circ\).
\end{assumption}
Under the basic nonnegativity assumption on edge costs, this assumption only rules out the degenerate zero-cost case. If \(C^{\SO}(\bm\mu)=0\) for some \(\bm\mu\in U^\circ\), then the system-optimal cost is already zero, and no positive efficiency loss caused by selfish routing is present. We henceforth restrict attention to the nondegenerate regime covered by the above positivity assumption. Under this assumption, the Price of Anarchy is defined on \(U^\circ\) by
\[
\PoA(\bm\mu):=\frac{C^{\UE}(\bm\mu)}{C^{\SO}(\bm\mu)}.
\]
Thus \(\PoA\) is a well-defined demand-dependent value function. The object studied in this paper is the full-dimensional demand map \(\PoA:U^\circ\to\R\), \(\bm\mu\longmapsto \PoA(\bm\mu)\).

\begin{remark}[Path-flow selections and value-level objects]
	Throughout the paper, only \(\bm\tau(\bm\mu)\), \(\bm\lambda(\bm\mu)\), \(C^{\UE}(\bm\mu)\), \(C^{\SO}(\bm\mu)\), and \(\PoA(\bm\mu)\) are treated as functions of the demand vector \(\bm\mu\). In general, Wardrop equilibrium path flows and system-optimal path flows may be set-valued. Unless a selection is explicitly specified, neither \(f^{*\UE}(\bm\mu)\) nor \(f^{*\SO}(\bm\mu)\) is treated as a single-valued function of \(\bm\mu\).
\end{remark}

\subsection{Active-Set Regions and Restricted Beckmann Problems}

This subsection introduces demand regions induced by active path sets and, for each fixed active path set, the corresponding restricted Beckmann problem. In the preceding subsection, Wardrop equilibria have been formulated as minimizers of the Beckmann problem. Hence, throughout this subsection, active path sets are understood with respect to Beckmann minimizers.

For any path flow \(f\in\mathcal F(\bm\mu)\), define its positive-flow support by \(\supp(f):=\{p\in\mathcal P:\ f_p>0\}\).
If \(f\) is a Wardrop equilibrium flow under demand \(\bm\mu\), then \(\supp(f)\) is called the active path set of this equilibrium flow. Since Wardrop equilibrium path flows need not be unique, the active path set is not treated as a single-valued object determined solely by \(\bm\mu\).

For each candidate active path set \(\alpha\subseteq\mathcal P\), define
\[
\Omega_\alpha
:=
\left\{
\bm\mu\in U^\circ:
\exists\ f^{*\UE}\in\argmin_{f\in\mathcal F(\bm\mu)}B(f)
\ \text{such that}\ 
\supp(f^{*\UE})=\alpha
\right\}.
\]
Since every \(\bm\mu\in U^\circ\) admits at least one Beckmann minimizer and \(\mathcal P\) is finite, we obtain the finite cover \(U^\circ=\bigcup_{\alpha\subseteq\mathcal P}\Omega_\alpha\).

Fix a candidate active path set \(\alpha\subseteq\mathcal P\). Write \(\alpha_i:=\alpha\cap\mathcal P_i\) for \(i=1,\dots,n\).
If \(\alpha_i=\varnothing\) for some \(i\), then \(\alpha\) cannot be the active path set associated with any positive demand vector in \(U^\circ\). We therefore consider only candidate active path sets satisfying \(\alpha_i\neq\varnothing\) for every \(i\). After deleting all paths outside \(\alpha\), let \(f_\alpha\in\R^{|\alpha|}\) denote the restricted path-flow vector, and let \(C_\alpha\) and \(R_\alpha\) denote the corresponding restricted edge--path incidence matrix and OD--path incidence matrix, respectively. For \(\bm\mu\in U^\circ\), define the restricted feasible set by \(\mathcal F_\alpha(\bm\mu):=\{f_\alpha\in\R_{\geq 0}^{|\alpha|}:\ R_\alpha f_\alpha=\bm\mu\}\).
The corresponding restricted Beckmann potential is \(B_\alpha(f_\alpha):=\sum_{e\in\mathcal E}\int_0^{(C_\alpha f_\alpha)_e}c_e(t)\,dt\).
Whenever \(\mathcal F_\alpha(\bm\mu)\neq\varnothing\), the associated restricted Beckmann value function is defined by \(V_\alpha(\bm\mu):=\min_{f_\alpha\in\mathcal F_\alpha(\bm\mu)}B_\alpha(f_\alpha)\).
In particular, if \(\bm\mu\in\Omega_\alpha\), then there exists a Beckmann minimizer of the original problem whose support is exactly \(\alpha\). Restricting this minimizer to \(\alpha\) gives a feasible solution of the restricted problem and attains the original Beckmann value. Hence \(V_\alpha(\bm\mu)=V(\bm\mu)\).

For any \(f_\alpha\in\mathcal F_\alpha(\bm\mu)\), denote the induced edge-load vector by \(x^\alpha(f_\alpha):=C_\alpha f_\alpha\).
For any path \(p\in\mathcal P_i\), regardless of whether \(p\in\alpha_i\), define its path cost under this restricted edge-load vector by \(\ell_p^\alpha(f_\alpha):=\sum_{e\in p}c_e\bigl(x_e^\alpha(f_\alpha)\bigr)\).
Thus, \(\bm\mu\in\Omega_\alpha\) is equivalent to the existence of \(f_\alpha\in\R^{|\alpha|}\) and \(\bm\lambda=(\lambda_1,\dots,\lambda_n)^\top\in\R^n\) such that
\[
f_\alpha>\bmzero,
\qquad
R_\alpha f_\alpha=\bm\mu,
\]
and, for every \(i=1,\dots,n\),
\[
\ell_p^\alpha(f_\alpha)=\lambda_i
\qquad
(p\in\alpha_i),
\qquad
\ell_q^\alpha(f_\alpha)\ge \lambda_i
\qquad
(q\in\mathcal P_i\setminus\alpha_i).
\]
Here \(\alpha\) records the strictly positive path-flow support; paths with zero flow but equal cost are not included in \(\alpha\).

The restricted Beckmann problem converts the local analysis associated with a fixed active path set into a value-function problem over a fixed collection of paths. This construction will serve as the basis for the subsequent local regularity analysis.

\subsection{Beckmann Value Functions and \(C^1\) Regularity}

This subsection develops the main \(C^1\) regularity theory of the paper. We first establish the regularity, with respect to the demand vector, of the original Beckmann value function \(V\), the restricted Beckmann value function \(V_\alpha\), and the SO value function \(C^{\SO}\). We then derive the local \(C^1\) regularity result on active-set interiors.

\begin{lemma}[\(C^1\) regularity of the Beckmann value function]
	\label{lem:value-C1}
	Assume that each edge cost function \(c_e:[0,+\infty)\to\R_+\) is continuous and nondecreasing. Then the original Beckmann value function \(V(\bm\mu):=\min_{f\in\mathcal F(\bm\mu)}B(f)\), \(\bm\mu\in U^\circ\), satisfies
	\[
	V\in C^1(U^\circ),
	\qquad
	D_{\bm\mu}V(\bm\mu)=\bm\lambda(\bm\mu).
	\]
	In particular, the OD equilibrium cost vector \(\bm\lambda\) is continuous on \(U^\circ\).
\end{lemma}

\begin{proof}{Proof}
	The proof is given in Appendix \ref{app:sec2-proofs}.
\end{proof}

The same value-function regularity result applies to the restricted Beckmann problem introduced in the preceding subsection. For any admissible active path set \(\alpha\subseteq\mathcal P\), the restricted Beckmann value function \(V_\alpha\) is \(C^1\) on the interior of its demand domain. Let \(\bm\lambda^\alpha(\bm\mu):=(\lambda_1^\alpha(\bm\mu),\dots,\lambda_n^\alpha(\bm\mu))^\top\) denote the OD equilibrium cost vector induced by the restricted Beckmann problem. Then \(D_{\bm\mu}V_\alpha(\bm\mu)=\bm\lambda^\alpha(\bm\mu)\).

If \(\bm\mu^0\in\intO(\Omega_\alpha)\), then there exists an open neighborhood \(N\subset U^\circ\) of \(\bm\mu^0\) such that \(N\subset\Omega_\alpha\). By the consistency between \(V_\alpha\) and \(V\) on \(\Omega_\alpha\), established in the preceding subsection, we have, for every \(\bm\mu\in N\), \(V(\bm\mu)=V_\alpha(\bm\mu)\) and \(\bm\lambda(\bm\mu)=\bm\lambda^\alpha(\bm\mu)\).

\begin{theorem}[Local \(C^1\) regularity of \(C^{\UE}\) on active-set interiors]
	\label{thm:local-C1-CUE}
	Assume that each edge cost function \(c_e\in C^1([0,+\infty))\). If \(\bm\mu^0\in\intO(\Omega_\alpha)\) for some admissible active path set \(\alpha\subseteq\mathcal P\), then there exists an open neighborhood \(N\subset U^\circ\) of \(\bm\mu^0\) such that \(C^{\UE}\in C^1(N)\).
\end{theorem}

\begin{proof}{Proof}
	The proof is given in Appendix \ref{app:sec2-proofs}.
\end{proof}

We next consider the SO total cost. Assume that each edge cost function \(c_e\in C^1([0,+\infty))\), and that \(x\mapsto x c_e(x)\) is convex. By the SO--Beckmann reduction in the preceding subsection, with \(\bar c_e(x):=c_e(x)+x c'_e(x)\), the SO value function can be written as the Beckmann value function of the transformed network:
\[
C^{\SO}(\bm\mu)=\min_{f\in\mathcal F(\bm\mu)}\sum_{e\in\mathcal E}\int_0^{(Cf)_e}\bar c_e(t)\,dt.
\]
The Beckmann value-function regularity result therefore implies \(C^{\SO}\in C^1(U^\circ)\).

Define the regular demand region by \(\mathcal O:=\bigcup_{\alpha\subseteq\mathcal P}\intO(\Omega_\alpha)\), and then we have the following \(C^1\) result.

\begin{corollary}[\(C^1\) regularity of \(\PoA\)]
	\label{C1_regularity_of_PoA}
	Assume that each edge cost function \(c_e\in C^1([0,+\infty))\), and that \(x\mapsto x c_e(x)\) is convex. Under \cref{ass:SO_Cost_positive}, one has \(\PoA\in C^1(\mathcal O)\).
\end{corollary}

When \(n=1\), the demand space reduces to \(U^\circ=(0,+\infty)\). If \(\mu^0\in\mathcal O=\bigcup_{\alpha\subseteq\mathcal P}\intO(\Omega_\alpha)\), then there exist an active path set \(\alpha\) and a neighborhood \(N\subset(0,+\infty)\) of \(\mu^0\) such that \(N\subset\Omega_\alpha\). Hence, for every demand in this neighborhood, there exists a Wardrop equilibrium flow whose support is \(\alpha\). The preceding theorem then yields \(\PoA\in C^1(N)\).

This conclusion is consistent with the local \(C^1\) regularity result of Cominetti--Dose--Scarsini for the single-OD case \cite{cds_mp2024}. In their formulation, regular demand levels correspond to demand values at which the equilibrium active path set does not change. In the notation of the present paper, such demand values are characterized by \(\mu^0\in\mathcal O\). Therefore, when specialized to \(n=1\), the theorem above yields the same local \(C^1\) behavior as their single-OD result.

Regarding the assumptions, the corresponding result of Cominetti--Dose--Scarsini typically requires the relevant edge costs to have strictly positive derivatives. By contrast, the present theorem treats value-level regularity through the Beckmann value function and the restricted Beckmann problem, and it does not require the UE-side condition \(c'_e(x)>0\) at the relevant edge loads. Thus, in the single-OD specialization, the present result applies under a broader set of conditions and can be viewed as a generalization of that single-OD local \(C^1\) regularity result.

\section{Distributional Properties of the Nondifferentiability Set}
\label{sec:nondiff}

\subsection{Exceptional Region and the Nondifferentiability Set}

We retain the notation introduced in the previous section. For each candidate active path set \(\alpha\subseteq\mathcal P\), let \(\Omega_\alpha\) denote the corresponding active-set region. Define
\[
\mathcal O:=\bigcup_{\alpha\subseteq\mathcal P}\operatorname{int}(\Omega_\alpha),
\qquad
\mathcal E:=U^\circ\setminus\mathcal O .
\]
Here \(\mathcal O\) is the regular region covered by the local \(C^1\) regularity theory developed in the preceding section, whereas \(\mathcal E\) is the exceptional region induced by switches or degeneracies of the active path set.

\begin{proposition}[First localization of nondifferentiability]
	\label{prop:first-localization}
	Under the conditions of \cref{C1_regularity_of_PoA}, \(\Nondiff(\PoA)\subseteq \mathcal E\).
\end{proposition}

This proposition gives a one-sided localization. It asserts that nondifferentiability of \(\PoA\) can occur only in the exceptional region. However, \(\mathcal E\) is still only an active-set-level candidate exceptional region, and it need not coincide with the actual nondifferentiability set.

We next derive a measure-theoretic conclusion for the actual nondifferentiability set from the regularity of the value function. Let the Beckmann value function be
\[
V(\bm\mu):=\min_{f\in\mathcal F(\bm\mu)}B(f),
\qquad
B(f)=\sum_{e\in\mathcal E}\int_0^{(Cf)_e}c_e(t)\,dt .
\]
By the global value-function result in the previous section, \(V\) is a finite convex function on \(U^\circ\), and \(\nabla_{\bm\mu}V(\bm\mu)=\bm\lambda(\bm\mu)\). By Alexandrov's theorem \cite{alexandrov1939}, \(V\) is twice differentiable almost everywhere on \(U^\circ\); in particular, the continuous monotone map \(\bm\lambda=\nabla V\) is differentiable almost everywhere (see also Mignot's theorem on the almost-everywhere differentiability of monotone maps \cite{mignot1976}). Consequently, \(C^{\UE}(\bm\mu)=\bm\mu^\top\bm\lambda(\bm\mu)\) is also differentiable almost everywhere. Under the hypotheses of \cref{C1_regularity_of_PoA}, \(C^{\SO}\) is \(C^1\) on \(U^\circ\), and therefore \(\PoA(\bm\mu)=C^{\UE}(\bm\mu)/C^{\SO}(\bm\mu)\) is differentiable almost everywhere on \(U^\circ\).

\begin{proposition}[Measure-zero property of the nondifferentiability set]
	\label{prop:nondiff-measure-zero}
	Under the conditions of \cref{C1_regularity_of_PoA}, \(\mathcal L^n\bigl(\Nondiff(\PoA)\bigr)=0\).
\end{proposition}

This measure-zero conclusion is not obtained by proving that \(\mathcal E\) itself has measure zero. In fact, even in the simplest single-OD setting, \(\mathcal E\) may have positive measure. The following example shows that switches of the active path set on a positive-measure set do not necessarily destroy smoothness at the cost level.

\begin{example}
	\label{ex:E_positive_measure}
	Let \(K\subset[0,1]\) be a fat Cantor set containing the endpoints, with \(|K|_1>0\), and suppose that \(K\) is closed and has empty interior. Set \(F:=\frac12K\subset[0,\tfrac12]\), and write the relatively open set \([0,\tfrac12]\setminus F\) as a countable disjoint union of open intervals \([0,\tfrac12]\setminus F=\bigsqcup_{n=1}^\infty (a_n,b_n)\). Let
	\[
	\psi(t)=
	\begin{cases}
		\exp\!\bigl(-1/(t(1-t))\bigr), & t\in(0,1),\\
		0, & t\notin(0,1),
	\end{cases}
	\]
	be the standard bump function. Fix \(\delta>0\) sufficiently small, define \(\varepsilon_n:=\delta\,2^{-n}(b_n-a_n)^n\), and set
	\[
	g(x):=\sum_{n=1}^{\infty}\varepsilon_n\,\psi\!\left(\frac{x-a_n}{b_n-a_n}\right),
	\qquad x\ge0 .
	\]
	By the standard bump construction, \(g\in C^\infty([0,+\infty))\). Taking \(\delta\) sufficiently small ensures that \(\|g'\|_\infty<\frac12\). Moreover, on \([0,\tfrac12]\), \(g(x)=0\) if and only if \(x\in F\), and \(g(x)>0\) if and only if \(x\in[0,\tfrac12]\setminus F\).
	Let \(c(x):=x+g(x)\). Then \(c\in C^\infty([0,+\infty))\), and \(c'(x)=1+g'(x)>\frac12\). Thus \(c\) is strictly increasing.
	
	Consider the following single-OD network in \cref{fig:E_positivemeasure}. The nodes are \(O,v_1,v_2,D\). The edge costs are as follows: the edge \(O\to v_1\) has cost \(x\), the edge \(v_1\to D\) has cost \(c(x)\), the edge \(O\to v_2\) has cost \(c(x)\), the edge \(v_2\to D\) has cost \(x\), and the edge \(v_1\to v_2\) has cost \(0\). Denote the three \(O\)-\(D\) paths by
	\[
	p_u:=O\to v_1\to D,\qquad
	p_\ell:=O\to v_2\to D,\qquad
	p_m:=O\to v_1\to v_2\to D .
	\]
	Let \(\alpha:=\{p_u,p_\ell\}\) and \(\beta:=\{p_u,p_\ell,p_m\}\).
	
	\begin{figure}[htp!]
		\FIGURE
		{\begin{tikzpicture}[
			>=Stealth,
			node distance=2.4cm,
			vertex/.style={circle,draw,minimum size=9mm,inner sep=0pt},
			edge/.style={->,line width=0.8pt}
			]
			\node[vertex] (O) at (0,0) {$O$};
			\node[vertex] (U) at (3,1.2) {$V_1$};
			\node[vertex] (L) at (3,-1.2) {$V_2$};
			\node[vertex] (D) at (6,0) {$D$};
			\draw[edge] (O) -- node[above] {$e_1$} (U);
			\draw[edge] (U) -- node[above] {$e_2$} (D);
			\draw[edge] (O) -- node[below] {$e_3$} (L);
			\draw[edge] (L) -- node[below] {$e_4$} (D);
			\draw[edge] (U) -- node[right] {$e_5$} (L);
		\end{tikzpicture}}
		{Network used in \cref{ex:E_positive_measure}. \label{fig:E_positivemeasure}}
		{}
	\end{figure}
	
	If the path flow is \((f_u,f_\ell,f_m)\) with \(f_u+f_\ell+f_m=\mu\), then the three path costs are
	\[
	L_u=(f_u+f_m)+c(f_u),\qquad
	L_\ell=c(f_\ell)+(f_\ell+f_m),\qquad
	L_m=\mu+f_m .
	\]

	For this network, on \((0,1)\) one has
	\[
	\Omega_\alpha\cap(0,1)=K\cap(0,1),
	\qquad
	\Omega_\beta\cap(0,1)=(0,1)\setminus K .
	\]
	
	Indeed, if \(\mu\in K\cap(0,1)\), take \(s=\mu/2\in F\) and consider the flow \(f_u=f_\ell=s,\ f_m=0\). Since \(g(s)=0\), all three paths have the same cost, and the positive-flow paths are precisely \(p_u\) and \(p_\ell\). Hence \(\mu\in\Omega_\alpha\). Conversely, if a Wardrop equilibrium has positive-flow support \(\alpha\), the equal-cost condition gives \(f_u=f_\ell=\mu/2\), while the requirement that \(p_m\) is not cheaper gives \(g(\mu/2)=0\). Hence \(\mu\in K\).
	
	If \(\mu\in(0,1)\setminus K\), define \(h(s):=2s+g(s)\). Since \(h'(s)=2+g'(s)>\frac32\), \(h\) is strictly increasing. Moreover, because \(g=0\) on \(F\), we have \(h(F)=2F=K\). Hence there is a unique \(s\in[0,\tfrac12]\setminus F\) such that \(\mu=h(s)\). Set \(f_u=f_\ell=s,\qquad f_m=g(s)>0.\) Then all three paths have the same cost and all three paths carry strictly positive flow; hence \(\mu\in\Omega_\beta\). The reverse inclusion follows directly from the same Wardrop conditions \eqref{eq:wardrop-equal}--\eqref{eq:wardrop-ineq}. Within \((0,1)\), no other strictly positive path-flow support can be realized by a Wardrop equilibrium.
	
	Thus the exceptional behavior on \(K\) is indeed generated by a switch of the active path set: when \(\mu\in K\cap(0,1)\), the path \(p_m\) has the same cost as the used paths but carries zero flow; whereas at any nearby \(\nu\notin K\), the corresponding equilibrium assigns strictly positive flow \(g(H(\nu))>0\) to \(p_m\), where \(H:=h^{-1}\).
	
	Since \(K\) has empty interior and \((0,1)\setminus K\) is open,
	\[
	\mathcal E\cap(0,1)=K\cap(0,1),
	\qquad
	|\mathcal E\cap(0,1)|_1=|K\cap(0,1)|_1>0 .
	\]
	
	On the other hand, define \(\Phi(s):=2s+2g(s)\). The equilibrium cost in this example satisfies \(\lambda(\mu)=\Phi(H(\mu)), \qquad H:=h^{-1}.\) Since \(g\in C^\infty\) and \(h'(s)>0\), we have \(H\in C^\infty((0,1))\), and hence \(\lambda\in C^\infty((0,1))\).
\end{example}

This example shows that the active-set exceptional region may have positive measure. At the same time, in this very example, the equilibrium cost \(\lambda\) still satisfies \(\lambda\in C^\infty((0,1))\). Therefore, switches of the active path set provide only a candidate localization for nondifferentiability; this information alone cannot characterize the exact structure of \(\Nondiff(\PoA)\).

\subsection{Countable Manifold Cover of Nondifferentiability Points}

This subsection further characterizes the local geometric structure of the nondifferentiability set. The preceding subsection shows that nondifferentiability can occur only outside the active-set interiors. The result below further shows that, under the strictly monotone/constant edge-cost structure commonly used in regularity analysis of traffic networks, the actual nondifferentiability set is not only contained in the exceptional set, but is also covered by at most countably many local codimension-one \(C^1\) submanifolds. This cost structure allows two types of edges: congestible edges with strictly positive marginal cost, and constant-cost edges whose costs do not vary with the load.

\begin{assumption}[Strictly monotone/constant edge-cost structure]
	\label{ass:positive-constant-costs}
	For every edge \(e\in\mathcal E\), assume that \(c_e\in C^1([0,+\infty))\), and that one of the following two alternatives holds:
	\[
	c'_e(t)>0\qquad \forall t\ge0,
	\]
	or there exists a constant \(\kappa_e\in\R_+\) such that \(c_e(t)\equiv \kappa_e\) for all \(t\ge0\).
\end{assumption}

For any \(\bm\mu^0\in U^\circ\), define \(\mathcal A(\bm\mu^0):=\{\alpha\subseteq\mathcal P:\ \bm\mu^0\in\operatorname{cl}_{U^\circ}(\Omega_\alpha)\}\), where \(\operatorname{cl}_{U^\circ}\) denotes closure relative to \(U^\circ\). Since \(\mathcal P\) is finite, the set \(\mathcal A(\bm\mu^0)\) is finite.

Under \cref{ass:SO_Cost_positive,ass:positive-constant-costs}, there exists a neighborhood \(W\subset U^\circ\) of \(\bm\mu^0\) such that, for every \(\alpha\in\mathcal A(\bm\mu^0)\), there exists a \(C^1\) map \(\widehat P_\alpha:W\to\R\) satisfying \(\widehat P_\alpha(\bm\mu)=\PoA(\bm\mu)\) for all \(\bm\mu\in W\cap\Omega_\alpha\).
Hence the local behavior of \(\PoA\) near \(\bm\mu^0\) can be represented, regime by regime, by finitely many \(C^1\) extensions.

\begin{lemma}[Local extension separation at nondifferentiability points]
	\label{lem:local-extension-separation}
	Suppose that \cref{ass:SO_Cost_positive,ass:positive-constant-costs} hold, and that \(xc_e(x)\) is convex for all \(e\). If \(\bm\mu^0\in\Nondiff(\PoA)\), then there exist \(\alpha,\beta\in\mathcal A(\bm\mu^0)\) and a neighborhood \(W\subset U^\circ\) of \(\bm\mu^0\) such that the corresponding local \(C^1\) extensions \(\widehat P_\alpha,\widehat P_\beta:W\to\R\) satisfy
	\[
	\widehat P_\alpha(\bm\mu^0)=\widehat P_\beta(\bm\mu^0)=\PoA(\bm\mu^0),
	\qquad
	\nabla \widehat P_\alpha(\bm\mu^0)\neq\nabla \widehat P_\beta(\bm\mu^0).
	\]
\end{lemma}

\begin{proof}{Proof}
	The proof is given in Appendix \ref{app:sec3-proofs}.
\end{proof}

\Cref{lem:local-extension-separation} shows that local nondifferentiability can only arise from a discrepancy between the first-order information of distinct \(C^1\) extensions. If all local extensions that agree with \(\PoA\) at \(\bm\mu^0\) have the same first-order expansion, then these local representations cannot generate nondifferentiability of \(\PoA\) at \(\bm\mu^0\). Conversely, whenever nondifferentiability occurs, there must exist two local extensions with the same value but different gradients at \(\bm\mu^0\). Thus, locally, a nondifferentiability point lies on the regular zero set of some difference \(\widehat P_\alpha-\widehat P_\beta \).

\begin{theorem}[At most countable \(C^1\) submanifold cover of the nondifferentiability set]
	\label{thm:nondiff-countable-manifold-cover}
	Under the assumptions of \cref{lem:local-extension-separation}, there exist at most countably many local \((n-1)\)-dimensional \(C^1\) submanifolds \(M_1,M_2,\dots\subset U^\circ\) such that
	\[
	\Nondiff(\PoA)\subseteq\bigcup_{\ell=1}^{\infty}M_\ell .
	\]
\end{theorem}

\begin{proof}{Proof}
	The proof is given in Appendix \ref{app:sec3-proofs}.
\end{proof}

\Cref{thm:nondiff-countable-manifold-cover} provides a geometric characterization of the nondifferentiability set of \(\PoA\) in \(U^\circ\). This result is consistent with the breakpoint results in the one-dimensional demand theory: in the single-OD case, active-set transitions appear as isolated breakpoints on the demand axis; in the one-parameter multi-OD case, the nondifferentiability set of the one-dimensional parametrized problem is contained in the intersection between the demand curve and these local \((n-1)\)-dimensional \(C^1\) submanifolds. Thus, the breakpoints in the existing one-dimensional theory can be viewed as traces of local \((n-1)\)-dimensional \(C^1\) submanifolds in \(U^\circ\) along lower-dimensional demand paths.

\section{Higher-Order Regularity: Obstructions and Sufficient Conditions}
\label{sec:higher-order}

\subsection{\texorpdfstring{\(C^1\)}{C1} Regularity versus \texorpdfstring{\(C^r\)}{Cr} Regularity}

The local \(C^1\) result in \cref{sec:model} is based on the convex-duality structure of the restricted Beckmann value function. More precisely, if \(\bm\mu^0\in\intO(\Omega_\alpha)\), then, in a neighborhood of \(\bm\mu^0\),
\[
V(\bm\mu)=V_\alpha(\bm\mu),
\qquad
D_{\bm\mu}V_\alpha(\bm\mu)=\bm\lambda^\alpha(\bm\mu).
\]
Thus the equilibrium cost vector of the original network locally coincides with that of the restricted network, and \(C^{\UE}(\bm\mu)=\bm\mu^\top\bm\lambda(\bm\mu)\in C^1\).
This conclusion only yields the continuity of \(\bm\lambda\); it does not imply the existence or continuity of
\(D\bm\lambda,D^2\bm\lambda,\dots,D^{r-1}\bm\lambda\). In fact, even when the demand point lies in the interior of an active-set region and all edge costs are smooth, the equilibrium cost need not be \(C^r\). The following example exhibits this phenomenon for \(r=3\).

\begin{example}
	\label{ex:Cr-degenerate-example}
	
	Consider the network shown in \cref{fig:Cr-degenerate-network}. It has three OD pairs, namely \((O,D)\), \((O,M)\), and \((M,D)\). The corresponding path sets are
	\[
	\mathcal P_1=\{p_{11},p_{12}\},\quad
	p_{11}=e_1\cup e_2,\quad p_{12}=e_3,
	\qquad
	\mathcal P_2=\{p_{21},p_{22}\},\quad
	p_{21}=e_1,\quad p_{22}=e_4,
	\]
	\[
	\mathcal P_3=\{p_{31},p_{32}\},\quad
	p_{31}=e_2,\quad p_{32}=e_5.
	\]
	The edge cost functions are
	\[
	c_1(t)=c_2(t)=(t-2)^3+8,\qquad
	c_3(t)=(t-1)^3+16,\qquad
	c_4(t)\equiv c_5(t)\equiv 100.
	\]
	
	\begin{figure}[htbp]
		\FIGURE
		{\begin{tikzpicture}[
			vertex/.style={circle, draw, thick, minimum size=0.9cm, inner sep=0pt},
			edge/.style={-Stealth, thick},
			elab/.style={font=\normalsize},
			clab/.style={font=\normalsize}
			]
			\node[vertex] (O) at (0,0) {\(O\)};
			\node[vertex] (M) at (3.6,1.25) {\(M\)};
			\node[vertex] (D) at (7.2,0) {\(D\)};
			\draw[edge] ($(O)+(0.42,0.10)$) -- node[above, pos=0.50, elab] {\(e_1\)} ($(M)+(-0.42,0.10)$);
			\draw[edge] ($(O)+(0.42,-0.10)$) -- node[below, pos=0.50, elab] {\(e_4\)} ($(M)+(-0.42,-0.10)$);
			\draw[edge] ($(M)+(0.42,0.10)$) -- node[above, pos=0.50, elab] {\(e_2\)} ($(D)+(-0.42,0.10)$);
			\draw[edge] ($(M)+(0.42,-0.10)$) -- node[below, pos=0.50, elab] {\(e_5\)} ($(D)+(-0.42,-0.10)$);
			\draw[edge] ($(O)+(0.43,-0.27)$) -- node[below, pos=0.50, elab] {\(e_3\)} ($(D)+(-0.43,-0.27)$);
		\end{tikzpicture}}
		{Network used in \cref{ex:Cr-degenerate-example}. \label{fig:Cr-degenerate-network}}
		{}
	\end{figure}
	
	Set \(\bm\mu^0=(2,1,1)\), and consider, at this demand, the path flow \(f_{11}^0=f_{12}^0=1\), \(f_{21}^0=1\), \(f_{31}^0=1\), and \(f_{22}^0=f_{32}^0=0\).
	Then \(p_{11}\) and \(p_{12}\) have the same equilibrium cost for \(\mathrm{OD}_1\), whereas \(p_{21}\) and \(p_{31}\) are strictly cheaper than \(p_{22}\) and \(p_{32}\), respectively. Hence the active path set at the base point is \(\alpha=\{p_{11},p_{12},p_{21},p_{31}\}\).

	Consider perturbations of \(\bm\mu^0\) of the form \(\bm\mu=(2+\xi,1+\eta_1,1+\eta_2)\).
	Write
	\[
	f_{11}=1+\frac{\xi}{2}+y,
	\qquad
	f_{12}=1+\frac{\xi}{2}-y,
	\qquad
	f_{21}=1+\eta_1,
	\qquad
	f_{31}=1+\eta_2,
	\]
	and keep \(f_{22}=f_{32}=0\).
	The corresponding edge loads are
	\[
	x_1=2+\frac{\xi}{2}+y+\eta_1,
	\qquad
	x_2=2+\frac{\xi}{2}+y+\eta_2,
	\qquad
	x_3=1+\frac{\xi}{2}-y .
	\]
	When the active set \(\alpha\) is fixed, the Wardrop condition for the first OD pair is \(c_1(x_1)+c_2(x_2)=c_3(x_3)\), or equivalently,
	\begin{equation}
		\left(\frac{\xi}{2}+y+\eta_1\right)^3
		+
		\left(\frac{\xi}{2}+y+\eta_2\right)^3
		-
		\left(\frac{\xi}{2}-y\right)^3
		=0.
		\label{eq:Cr-degenerate-full}
	\end{equation}
	For all sufficiently small \((\xi,\eta_1,\eta_2)\), equation \eqref{eq:Cr-degenerate-full} has a unique real root in \(y\), denoted by \(y=y(\xi,\eta_1,\eta_2)\) with \(y(0,0,0)=0\).
	By continuity, in a sufficiently small neighborhood one has \(f_{11}>0\), \(f_{12}>0\), \(f_{21}>0\), and \(f_{31}>0\).
	Moreover, the costs of \(p_{22}\) and \(p_{32}\) remain strictly larger than those of \(p_{21}\) and \(p_{31}\), respectively. Therefore \(\bm\mu^0\in\intO(\Omega_\alpha)\).
	
	If \(\lambda_1\) were \(C^3\) at \(\bm\mu^0\), then its restriction to any affine subspace would also be \(C^3\). It is therefore enough to examine regularity on the subspace \(\xi=0\). On this subspace, \eqref{eq:Cr-degenerate-full} reduces to
	\begin{equation}
		(y+\eta_1)^3+(y+\eta_2)^3+y^3=0.
		\label{eq:Cr-degenerate-slice}
	\end{equation}
	\Cref{eq:Cr-degenerate-slice} is homogeneous of degree three; hence its unique root satisfies \(y(t\eta_1,t\eta_2)=t\,y(\eta_1,\eta_2)\).

	The equilibrium cost of the first OD pair is
	\begin{equation}
	\lambda_1(\eta_1,\eta_2)=c_3(1-y(\eta_1,\eta_2))=16-y(\eta_1,\eta_2)^3 .
	\label{eq:Cr-lambda1}
	\end{equation}
	Let \(\Lambda(\eta_1,\eta_2):=\lambda_1(\eta_1,\eta_2)-16\).
	By \eqref{eq:Cr-lambda1} and the degree-one homogeneity of \(y\) established above, \(\Lambda\) is homogeneous of degree three. If \(\lambda_1\) were \(C^3\) at the origin, then \(\Lambda\) would have to be a homogeneous cubic polynomial. By symmetry in \(\eta_1\) and \(\eta_2\), and by evaluating along the directions \((1,0)\), \((1,1)\), and \((2,1)\), one directly verifies that no such homogeneous cubic polynomial can represent \(\Lambda\). Hence \(\lambda_1\) is not \(C^3\) at the base point.
\end{example}

This example shows that, even when all relevant active edge costs are \(C^\infty\) and strictly increasing, and even when
\(\bm\mu^0\) lies in the interior of an active-set region, the equilibrium cost vector need not have higher-order regularity. Consequently, the assumptions used in \cref{sec:model} to establish local \(C^1\) regularity are insufficient to imply local \(C^r\) regularity.

\subsection{\(C^r\) Regularity in the Multivariate Setting}

The example in the preceding subsection shows that the \(C^1\) regularity result in \cref{sec:model} cannot be automatically upgraded to \(C^r\) regularity solely from the higher-order smoothness of the edge cost functions. In this subsection, we consider a common class of traffic networks in which each edge cost function either has a strictly positive derivative on its domain or is constant; equivalently, the network satisfies \cref{ass:positive-constant-costs}. For this class of networks, we establish cellwise \(C^r\) regularity of the equilibrium cost vector, the UE cost, the SO cost, and \(\PoA\) in the multivariate demand setting.

\begin{theorem}[Cellwise \(C^r\) regularity of the UE cost]
	\label{thm:Cr_regularity_UEcost}
	Suppose that \(r\ge1\), and that \(c_e\in C^r([0,+\infty))\) for every edge \(e\). Under \cref{ass:positive-constant-costs}, if \(\bm\mu^0\in\intO(\Omega_\alpha)\), then there exists a neighborhood \(V\subset U^\circ\) of \(\bm\mu^0\) such that
	\[
	\bm\tau(\bm\mu)\in C^r(V),\qquad
	\bm\lambda(\bm{\mu})\in C^r(V),\qquad
	C^{\UE}(\bm{\mu})\in C^r(V).
	\]
\end{theorem}

\begin{proof}{Proof}
	The proof is given in Appendix \ref{app:sec4-proofs}.
\end{proof}

On the basis of \cref{thm:Cr_regularity_UEcost}, the SO side can be handled through the transformed network introduced in \cref{prop:BMW_for_SOcost}. The transformed edge costs \(\bar c_e(x)=c_e(x)+x c'_e(x)\) belong to \(C^{r-1}([0,+\infty))\). Assume in addition that, for every edge \(e\) whose cost function is not constant, \((xc_e(x))''>0\) for all \(x\geq0\). Then the transformed network satisfies \cref{ass:positive-constant-costs}: for such an edge, \(\bar c'_e(x)=(xc_e(x))''>0\) for all \(x\geq0\), while for a constant-cost edge, \(\bar c_e\) remains the same constant. Hence, for any \(\bm\mu^0\in\intO(\widetilde\Omega_\beta)\), where \(\widetilde\Omega_\beta\) denotes the region induced by the active path set \(\beta\) in the transformed network, \cref{thm:Cr_regularity_UEcost} applies to the transformed network with regularity \(r-1\). It follows that there exists a neighborhood \(V\subset U^\circ\) of \(\bm\mu^0\) such that the equilibrium cost vector of the transformed network belongs to \(C^{r-1}(V)\).

On the other hand, \(C^{\SO}\) is precisely the Beckmann value function of the transformed network. By the value-function differentiability result in \cref{sec:model}, \(D_{\bm\mu}C^{\SO}(\bm\mu)\) is equal to the equilibrium cost vector of the transformed network. Therefore, \(C^{\SO}\in C^r(V)\).

Then the following result can be derived.

\begin{corollary}[Cellwise \(C^r\) regularity of \(\PoA\)]
	\label{cor:cr-poa}
	Suppose that \(r\ge2\), and that \(c_e\in C^r([0,+\infty))\) for every edge \(e\). Assume that the original network satisfies \cref{ass:positive-constant-costs}, and that, for every edge \(e\) whose cost function is not constant, \((xc_e(x))''>0\) for all \(x\geq0\).
	If \(\bm\mu^0\in\intO(\Omega_\alpha)\cap\intO(\widetilde\Omega_\beta)\), then there exists a neighborhood \(V\subset U^\circ\) of \(\bm\mu^0\) such that \(\PoA\in C^r(V)\).
\end{corollary}

\begin{remark}[Covering structure of higher-order singular points]
	\label{rmk:Cr-nondiff-manifold}
	For \(s\ge1\), define
	\[
	Q_s:=\{\bm\mu\in U^\circ:\ \PoA\text{ is not }C^s\text{ on any neighborhood of }\bm\mu\}.
	\]
	The at most countable family of local \((n-1)\)-dimensional \(C^1\) manifolds used in \cref{sec:nondiff} to cover \(Q_1\) comes from the active-set exceptional region of the original network; at the cost level, the corresponding source of failure is the nondifferentiability of \(\bm\lambda\). For \(Q_2\), one must also include the corresponding family associated with the transformed network. Indeed, \(C^{\SO}\) is the Beckmann value function of the transformed network, and \(D_{\bm\mu}C^{\SO}\) is given by the equilibrium cost vector of the transformed network. Hence the nondifferentiability of \(D_{\bm\mu}C^{\SO}\) is controlled by the active-set exceptional region of the transformed network.
	
	Under the \(C^r\) conditions in this subsection, if \(\bm\mu^0\in\intO(\Omega_\alpha)\cap\intO(\widetilde\Omega_\beta)\), then \(\PoA\) is \(C^r\) on some neighborhood of \(\bm\mu^0\). Consequently, for every \(2\le s\le r\),
	\[
	Q_s\subseteq\left(\bigcup_{\ell\ge1}M_\ell^{\UE}\right)\cup\left(\bigcup_{\ell\ge1}M_\ell^{\SO}\right),
	\]
	where \(\{M_\ell^{\UE}\}_{\ell\ge1}\) is the covering family obtained in \cref{thm:nondiff-countable-manifold-cover} from the original network, and \(\{M_\ell^{\SO}\}_{\ell\ge1}\) is from the transformed network. Thus, when \(s>2\), higher-order singular points do not require any additional family of covering manifolds.
\end{remark}

\section{Affine Networks and a Boundary Maximum Principle for the PoA}
\label{sec:affine}

\subsection{Affine Cells and Polyhedral Exceptional Sets}

We now specialize the preceding theory to the affine-cost setting. Assume that every edge cost is of the form
\(c_e(x)=a_ex+b_e\), where \(a_e,b_e\ge0\). Let
\(A:=\operatorname{diag}(a_e)_{e\in\mathcal E}\) and
\(\bm b:=(b_e)_{e\in\mathcal E}\). Then
\[
B(f)=\frac12(Cf)^\top A(Cf)+\bm b^\top Cf,
\qquad
\Cost(f)=(Cf)^\top A(Cf)+\bm b^\top Cf .
\]
The SO problem is equivalently the Wardrop problem on the transformed network whose edge costs are given by \(\bar c_e(x)=c_e(x)+xc'_e(x)=2a_ex+b_e\).
Thus affine costs retain their affine structure on both the UE side and the SO side: an edge is of positive-slope type when \(a_e>0\), and of constant type when \(a_e=0\). The cellwise \(C^r\) regularity results established above therefore yield \(C^\infty\) regularity on active-set interiors in the affine case.

Fix a UE active path set \(\alpha\). After deleting all paths outside \(\alpha\), the KKT system for the UE problem is
\[
C_\alpha^\top(AC_\alpha f_\alpha+\bm b)=R_\alpha^\top\bm\lambda,
\qquad
R_\alpha f_\alpha=\bm\mu .
\]
For a fixed active path set, this system is linear in \((f_\alpha,\bm\lambda,\bm\mu)\). Hence, on the corresponding affine cell, \(\bm\lambda(\bm\mu)\) is affine in \(\bm\mu\), and consequently
\(C^{\UE}(\bm\mu)=\bm\mu^\top\bm\lambda(\bm\mu)\) is quadratic in \(\bm\mu\). The same argument applied to the transformed network, with \(a_e\) replaced by \(2a_e\), shows that \(C^{\SO}(\bm\mu)\) is also quadratic on the corresponding SO affine cell. Hence, on any UE/SO joint affine cell, \(\PoA\) is the quotient of two quadratic functions.

We next describe the geometry of the exceptional sets in the affine case. Define
\[
\mathcal O^{\UE}:=\bigcup_\alpha\intO(\Omega_\alpha),
\qquad
\mathcal E^{\UE}:=U^\circ\setminus\mathcal O^{\UE},
\]
where \(\Omega_\alpha\) is the region induced by the UE active path set \(\alpha\) in the original network. For the transformed network, define analogously
\[
\mathcal O^{\SO}:=\bigcup_\beta\intO(\widetilde\Omega_\beta),
\qquad
\mathcal E^{\SO}:=U^\circ\setminus\mathcal O^{\SO},
\]
where \(\widetilde\Omega_\beta\) is the region induced by the active path set \(\beta\) in the transformed network. The sets \(\mathcal E^{\UE}\) and \(\mathcal E^{\SO}\) are the active-set exceptional sets associated with the original network and the transformed network, respectively.

A subset \(P\subset\R^n\) is called a polyhedral cell if it is the relative interior of a nonempty convex polyhedron. Equivalently, after folding the implicit equalities into the equality system, there exist finitely many affine equalities and strict affine inequalities such that \(P=\{\bm\mu\in\R^n:G\bm\mu=g,\ D\bm\mu<d\}\).

\begin{proposition}[Polyhedral structure of exceptional sets in the affine case]
	\label{prop:affine-exceptional-polyhedral}
	Assume that every edge cost is affine. Then both \(\mathcal E^{\UE}\) and \(\mathcal E^{\SO}\) are finite unions of polyhedral cells of dimension at most \(n-1\). Moreover,
	\[
	\mathcal E^{\UE}=2\mathcal E^{\SO}:=\{2\bm\mu:\bm\mu\in\mathcal E^{\SO}\}.
	\]
\end{proposition}

\begin{proof}{Proof}
	The proof is given in Appendix \ref{app:sec5-proofs}.
\end{proof}

This proposition shows that, in affine networks, the countable manifold covering of the singular set obtained above reduces to a finite polyhedral-cell structure. Together with \cref{rmk:Cr-nondiff-manifold}, it gives \(Q_1\subseteq \mathcal E^{\UE}\) and \(Q_s\subseteq \mathcal E^{\UE}\cup\mathcal E^{\SO}\) for all \(s\ge2\).
Thus, in the affine case, \(Q_1\) and all higher-order singular candidate sets are controlled by finitely many lower-dimensional polyhedral cells.

We close this subsection with a boundary maximum principle for the PoA on affine cells. In the single-OD case, the maximum of the PoA over a demand interval is attained at an endpoint of the interval, and hence at a breakpoint \cite{cds_mp2024}. The following proposition is the higher-dimensional form of this principle: on every affine cell, the cellwise maximum of the PoA is attained on the cell boundary. 

\begin{proposition}[Boundary maximum principle for the PoA on affine cells]
	\label{thm:poa-cell-boundary-maximum}
	Assume that every edge cost is affine. Let \(K\subset U^\circ\) be a nonempty full-dimensional open convex polyhedron contained in a single UE active-set region \(\Omega_\alpha\), with \(C^{\SO}>0\) on its closure. Then the maximum of \(\PoA\) over \(\overline K\) is attained on the boundary of \(K\).
\end{proposition}

\begin{proof}{Proof}
	The proof is given in Appendix \ref{app:sec5-proofs}.
\end{proof}

\begin{remark}[The one-dimensional specialization]
	When \(n=1\), the cell \(K\) is an open interval and \(\partial K\) consists of its endpoints; \cref{thm:poa-cell-boundary-maximum} then asserts that the maximum of the PoA over a cell closure is attained at an endpoint, so the global maximum over the demand half-line is attained at a breakpoint. The proposition is thus exactly the higher-dimensional generalization of the single-OD endpoint maximum principle of \citeauthor{cds_mp2024}~\cite{cds_mp2024}.
\end{remark}

\subsection{An Example of a Multi-Commodity Affine Network}

We next present a three-commodity affine-network example. On the one hand, this example explicitly illustrates the quadratic value functions and the quadratic-fractional PoA expression on a fixed UE/SO joint affine cell in affine networks, as discussed in the preceding subsection. On the other hand, it illustrates the boundary maximum principle of \cref{thm:poa-cell-boundary-maximum}: the cellwise maximum of the PoA is attained in the relative interior of a boundary edge of the cell closure.

\begin{example}[A three-commodity affine network]
	\label{ex:three-commodity-affine}
	Consider the directed network shown in \cref{fig:affine-example-network}, taken from \citeauthor{fisk1979}~\cite{fisk1979} and recently revisited by \citeauthor{cds_pt2024}~\cite[Example~4.3]{cds_pt2024}, who analyze it along a one-parameter demand curve; here we study the full three-dimensional demand cone. The three edge costs are given by
	\[
	c_{e_1}(x)=x,\qquad
	c_{e_2}(x)=x,\qquad
	c_{e_3}(x)=x+90.
	\]
	The network has three OD pairs, with demands \((A,B):\mu_1\), \((A,C):\mu_2\), and \((B,C):\mu_3\).
	The corresponding path sets are \(\mathcal P_1=\{p_{11}\}\), \(\mathcal P_2=\{p_{21},p_{22}\}\), and \(\mathcal P_3=\{p_{31}\}\), where
	\[
	p_{11}=e_1,\qquad
	p_{21}=e_3,\qquad
	p_{22}=e_1\cup e_2,\qquad
	p_{31}=e_2.
	\]
	Thus, only the second OD pair has a routing choice.
	
	\begin{figure}[htbp]
		\FIGURE
		{\begin{tikzpicture}[>=Stealth,scale=1]
			\node[circle,draw,inner sep=2pt,minimum size=0.9cm] (A) at (0,0) {\(A\)};
			\node[circle,draw,inner sep=2pt,minimum size=0.9cm] (B) at (3,1.35) {\(B\)};
			\node[circle,draw,inner sep=2pt,minimum size=0.9cm] (C) at (6,0) {\(C\)};
			\draw[->,thick] (A) -- node[above left] {\(e_1\)} (B);
			\draw[->,thick] (B) -- node[above right] {\(e_2\)} (C);
			\draw[->,thick] (A) -- node[below] {\(e_3\)} (C);
		\end{tikzpicture}}
		{Network structure in \cref{ex:three-commodity-affine}. \label{fig:affine-example-network}}
		{}
	\end{figure}

	Define \(Q_I(\bm\mu):=(\mu_1+\mu_2)^2+(\mu_2+\mu_3)^2\), \(Q_D(\bm\mu):=\mu_1^2+\mu_2^2+\mu_3^2+90\mu_2\), and
	\[
	Q_M^{\UE}(\bm\mu)
	:=
	\frac{2}{3}
	\left(
	\mu_1^2+\mu_2^2+\mu_3^2
	+\mu_1\mu_2+\mu_2\mu_3-\mu_1\mu_3
	+45\mu_1+90\mu_2+45\mu_3
	\right).
	\]
	The SO value function induced by the corresponding transformed-network split cell is \(Q_M^{\SO}(\bm\mu):=Q_M^{\UE}(\bm\mu)-675\).
	Let
	\[
	\alpha_I:=\{p_{11},p_{22},p_{31}\},
	\qquad
	\alpha_M:=\{p_{11},p_{21},p_{22},p_{31}\},
	\qquad
	\alpha_D:=\{p_{11},p_{21},p_{31}\},
	\]
	denote the different active path sets. On the SO side, we always understand its cell structure through the transformed network; the transformed active set in \cref{tab:affine-cells} refers to the active path set of the transformed network. Direct computation gives the cell structure of this example, as shown in \cref{tab:affine-cells}. The inequalities describe relative interiors; the corresponding value functions extend continuously to the corresponding closures.

	\begin{table}[!htbp]
	\TABLE
	{Cell structure and cellwise value functions in \cref{ex:three-commodity-affine}.\label{tab:affine-cells}}
	{\footnotesize
	\setlength{\tabcolsep}{4pt}
	\renewcommand{\arraystretch}{1.35}
	\begin{tabular}{@{}ccccccc@{}}
	\hline\up
	Cell & Demand range & UE active set & transformed active set & \(C^{\UE}\) & \(C^{\SO}\) & \(\PoA\) \\ \hline\up
	\(\mathcal C_1\) & \(\mu_1+\mu_3+2\mu_2<45\) & \(\alpha_I\) & \(\alpha_I\) & \(Q_I\) & \(Q_I\) & \(1\) \\
	\(\mathcal C_2\) & \(\mu_1+\mu_3-\mu_2<45<\mu_1+\mu_3+2\mu_2<90\) & \(\alpha_I\) & \(\alpha_M\) & \(Q_I\) & \(Q_M^{\SO}\) & \(\dfrac{Q_I}{Q_M^{\SO}}\) \\
	\(\mathcal C_3\) & \(45<\mu_1+\mu_3-\mu_2,\ \mu_1+\mu_3+2\mu_2<90\) & \(\alpha_I\) & \(\alpha_D\) & \(Q_I\) & \(Q_D\) & \(\dfrac{Q_I}{Q_D}\) \\
	\(\mathcal C_4\) & \(\mu_1+\mu_3-\mu_2<45,\ \mu_1+\mu_3+2\mu_2>90\) & \(\alpha_M\) & \(\alpha_M\) & \(Q_M^{\UE}\) & \(Q_M^{\SO}\) & \(\dfrac{Q_M^{\UE}}{Q_M^{\SO}}\) \\
	\(\mathcal C_5\) & \(45<\mu_1+\mu_3-\mu_2<90,\ \mu_1+\mu_3+2\mu_2>90\) & \(\alpha_M\) & \(\alpha_D\) & \(Q_M^{\UE}\) & \(Q_D\) & \(\dfrac{Q_M^{\UE}}{Q_D}\) \\
	\(\mathcal C_6\) & \(\mu_1+\mu_3-\mu_2>90\) & \(\alpha_D\) & \(\alpha_D\) & \(Q_D\) & \(Q_D\) & \(1\) \down\\ \hline
	\end{tabular}}
	{}
	\end{table}

	\Cref{tab:affine-cells} gives the complete cellwise value expressions in this example. The UE transition candidates lie on the affine hyperplanes
	\[
	\mu_1+\mu_3+2\mu_2=90,
	\qquad
	\mu_1+\mu_3-\mu_2=90,
	\]
	whereas the transformed-network transition candidates lie on the affine hyperplanes
	\[
	\mu_1+\mu_3+2\mu_2=45,
	\qquad
	\mu_1+\mu_3-\mu_2=45.
	\]
	This explicitly exhibits the affine structure in the preceding subsection: on the relative interior of each cell, both \(C^{\UE}\) and \(C^{\SO}\) are quadratic functions of \(\bm\mu\), and \(\mathcal{E}^{\UE}=2\mathcal{E}^{\SO}\).

	We further consider the closure of the joint cell \(\mathcal C_5\):
	\[
	P=\left\{\bm\mu\in\R_+^3:\ \mu_1+\mu_3+2\mu_2\ge90,\ 45\le \mu_1+\mu_3-\mu_2\le90\right\}.
	\]
	On \(P\), \(C^{\UE}(\bm\mu)=Q_M^{\UE}(\bm\mu)\) and \(C^{\SO}(\bm\mu)=Q_D(\bm\mu)\), and hence
	\begin{equation}
	\PoA(\bm\mu)
	=
	\frac{
		2\left(
		\mu_1^2+\mu_2^2+\mu_3^2
		+\mu_1\mu_2+\mu_2\mu_3-\mu_1\mu_3
		+45\mu_1+90\mu_2+45\mu_3
		\right)
	}{
		3\left(
		\mu_1^2+\mu_2^2+\mu_3^2+90\mu_2
		\right)
	}.
	\label{eq:affine-poa-cell}
	\end{equation}
	Direct maximization of \eqref{eq:affine-poa-cell} over this cell closure gives
	\[
	\operatorname*{arg\,max}_{\bm\mu\in P}\PoA(\bm\mu)=\{(30,15,30)\},
	\qquad
	\max_{\bm\mu\in P}\PoA(\bm\mu)=\frac65.
	\]
	
	Note that \((30,15,30)\) lies in the relative interior of the one-dimensional edge of \(P\) with endpoints \((0,15,60)\) and \((60,15,0)\), hence on the boundary of the cell closure, in accordance with the boundary maximum principle of \cref{thm:poa-cell-boundary-maximum}.
\end{example}


\section{Numerical Illustrations}
\label{sec:numerical}

We illustrate the regularity theory on the Sioux Falls benchmark network
(LeBlanc et al.~\cite{leb1975}; Transportation Networks for
Research~\cite{transportationnetworks}), which has $24$ nodes, $76$ directed
links, and $528$ origin--destination (OD) pairs with baseline demand
$\bar\mu$ of total $360600$. We consider two link-cost models: the BPR
costs $c_e(x)=t_e^0\bigl(1+0.15\,(x/\kappa_e)^{4}\bigr)$ provided with the
data set, and the affine specialization $c_e(x)=t_e^0(1+0.15\,x/\kappa_e)$
obtained by setting the exponent to $1$, which meets the
strictly-increasing structure required by the affine analysis. User
equilibria and system optima are computed by a path-based simplicial
decomposition with an accelerated projected-gradient inner solver; every
reported run attains a relative Wardrop gap below $10^{-12}$, and the
baseline equilibrium agrees with the reference solution of
Transportation Networks for Research~\cite{transportationnetworks} to a maximum link-flow relative error of
$5.5\times10^{-12}$. Since the demand cone has dimension $528$, we display
the price of anarchy on slices spanned by three representative OD pairs,
$10{\to}16$, $22{\to}10$ and $20{\to}17$, keeping all remaining demands at
baseline; the theory developed in the previous sections applies on such
general-position slices.

\subsection{Verification of the value-function gradient}

The cornerstone of the framework is the $C^1$ regularity of the Beckmann
value function with gradient $D_\mu V=\lambda$. We verify this identity
numerically at the baseline demand: for each of the three varied OD pairs,
the central difference of $V$ along the corresponding demand coordinate is
compared with the equilibrium OD cost $\lambda_i$.
\cref{tab:gradient} reports the relative errors, which are of order
$10^{-9}$ under BPR costs and $10^{-11}$ under affine costs, consistently
with the theorem.

\begin{table}[htbp]
\TABLE
{Verification of $D_\mu V=\lambda$ at the baseline demand:
  relative error of the central-difference derivative against the
  equilibrium OD cost $\lambda_i$.\label{tab:gradient}}
{\begin{tabular}{@{}cccc@{}}
\hline\up
OD pair & $\lambda_i$ & rel.\ error (BPR) & rel.\ error (affine) \\ \hline\up
$10\to16$ & $20.084810$ & $5.6\times10^{-9}$ & $6.6\times10^{-11}$ \\
$22\to10$ & $22.868728$ & $3.5\times10^{-9}$ & $8.4\times10^{-12}$ \\
$20\to17$ & $16.927523$ & $1.4\times10^{-9}$ & $6.2\times10^{-11}$ \down\\ \hline
\end{tabular}}
{}
\end{table}

\subsection{Geometry of the nondifferentiability set}
\label{sec:num-geom}

We first describe the shape of the nondifferentiability set of $\PoA$
under the BPR costs. \cref{fig:poa1d} shows $\PoA$ along the three
one-dimensional demand lines: the function is smooth between isolated
kinks, and each kink coincides with a change of the active path set. An
exact homotopy continuation along the demand ray $\mu=t\bar\mu$, in which
the equilibrium conditions are linear for fixed active set, locates $247$
breakpoints; a one-sided-derivative test on the closed forms shows that
essentially all of them are genuine nondifferentiable points of $\PoA$.

\begin{figure}[!htbp]
\FIGURE
{\includegraphics[width=0.9\linewidth]{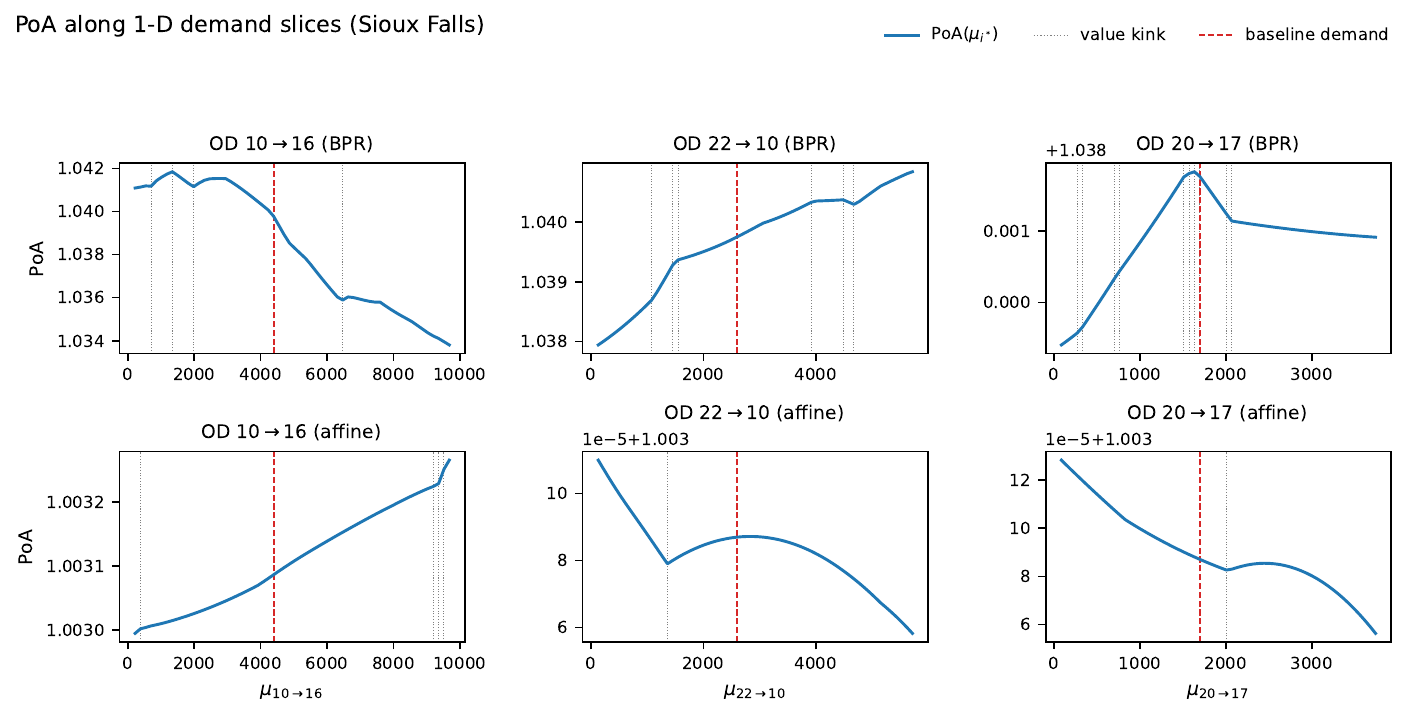}}
{$\PoA$ along one-dimensional demand slices (Sioux Falls).
  Dotted gray: detected value kinks; dashed red: baseline demand.
  \label{fig:poa1d}}
{}
\end{figure}

The same set is displayed on two- and three-dimensional slices in
\cref{fig:nondiff_bpr}: it organizes into \emph{curved} sheets --- a
collection of codimension-one surfaces in the three-OD demand subspace,
with the slice traces shown in the remaining cells. This is precisely the
content of the geometric characterization established earlier: under
strictly increasing smooth costs the nondifferentiability set is covered
by countably many codimension-one $C^1$ submanifolds.

\begin{figure}[!htbp]
\FIGURE
{\begin{tabular}{cc}
    \includegraphics[width=0.36\linewidth]{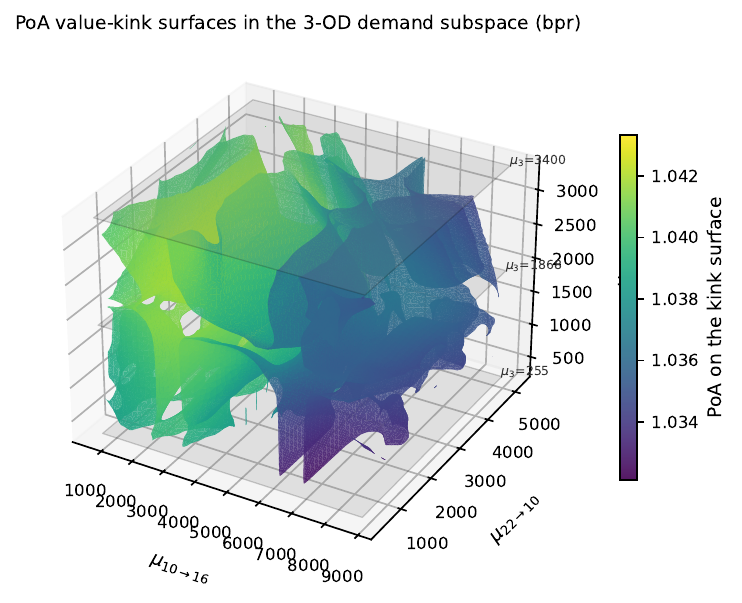} &
    \includegraphics[width=0.36\linewidth]{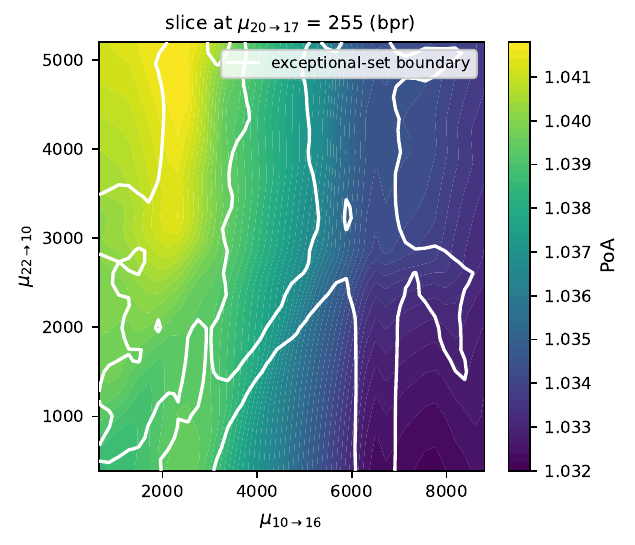}\\
    \includegraphics[width=0.36\linewidth]{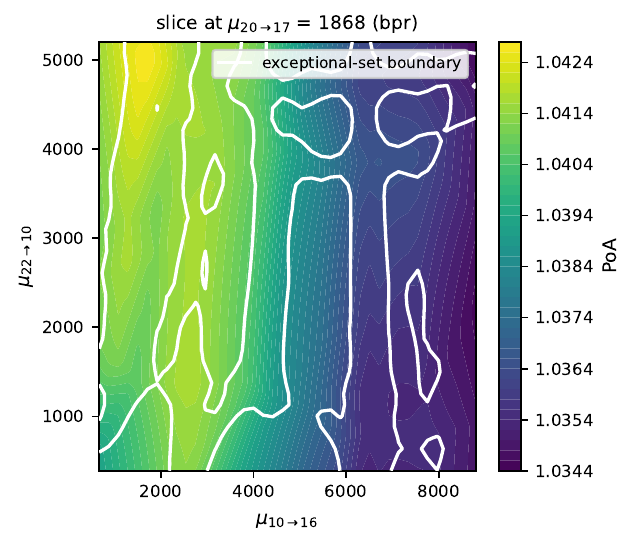} &
    \includegraphics[width=0.36\linewidth]{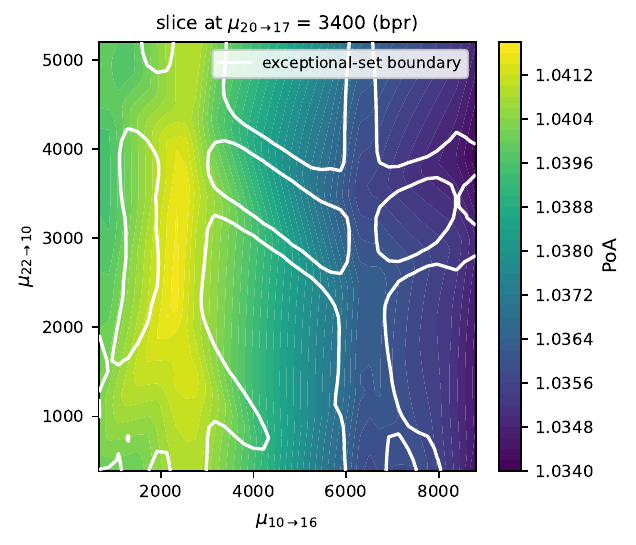}\\
  \end{tabular}}
{The nondifferentiability set under BPR costs. Top left: the
  value-kink surfaces in the three-OD demand subspace (the translucent
  planes mark the slice levels); remaining cells: the slice traces at
  three $\mu_{20\to17}$ levels.
  \label{fig:nondiff_bpr}}
{}
\end{figure}

\subsection{Planar locus and exact laws in the affine specialization}
\label{sec:num-affine}

Under affine costs the same nondifferentiability set becomes
\emph{planar}: \cref{fig:affine_planes} shows that the value-kink surfaces
of \cref{fig:nondiff_bpr} flatten into planar patches, and the slice
traces become straight lines, as the affine analysis asserts.

\begin{figure}[!htbp]
\FIGURE
{\begin{tabular}{cc}
    \includegraphics[width=0.36\linewidth]{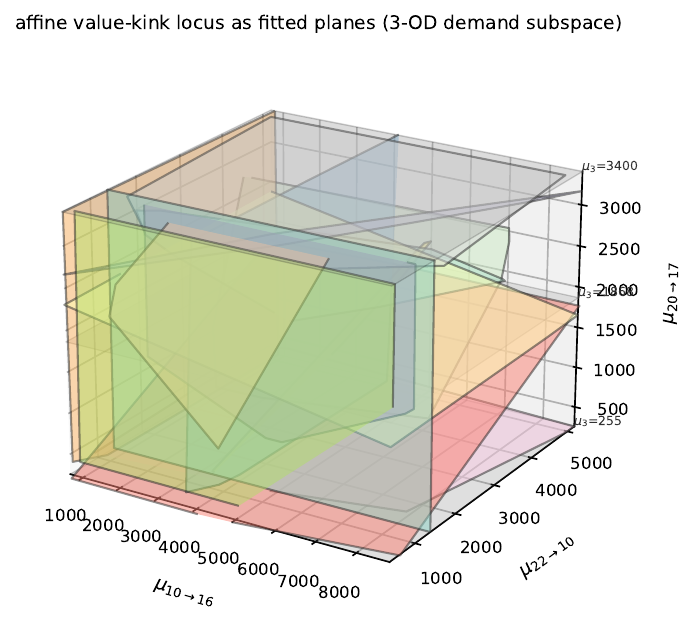} &
    \includegraphics[width=0.36\linewidth]{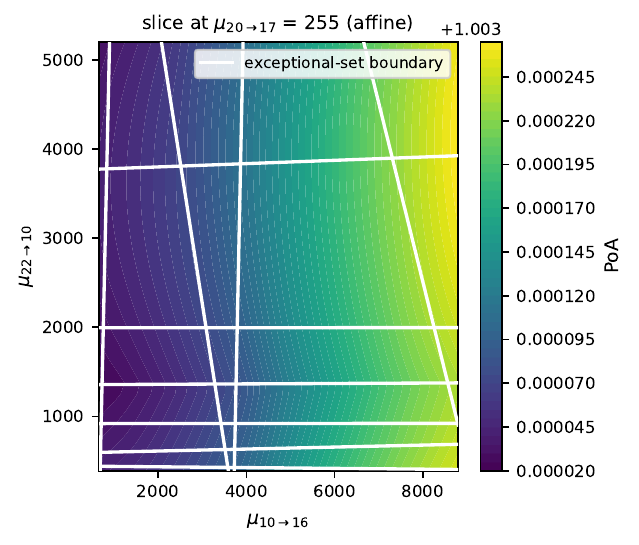}\\
    \includegraphics[width=0.36\linewidth]{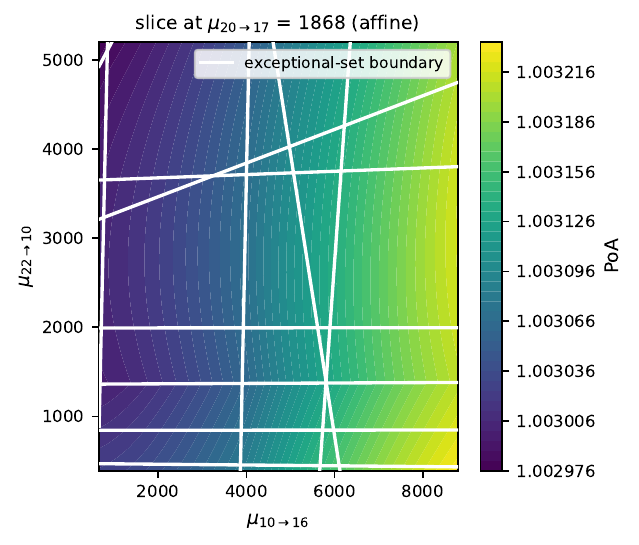} &
    \includegraphics[width=0.36\linewidth]{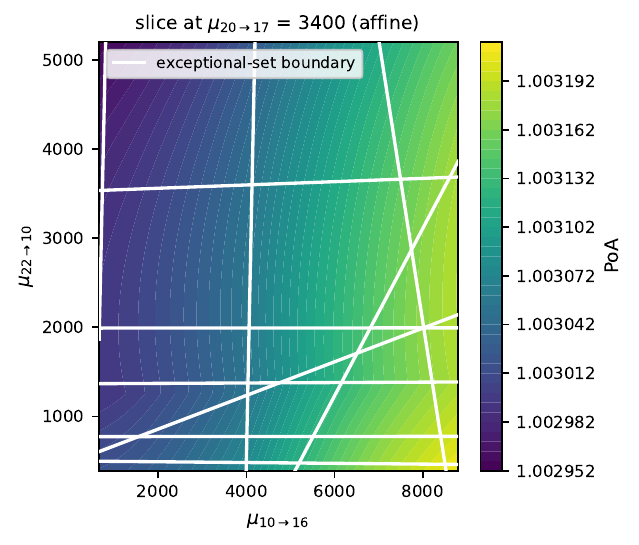}\\
  \end{tabular}}
{The nondifferentiability set under affine costs (same layout as
  \cref{fig:nondiff_bpr}): the surfaces and slice traces are planar.
  \label{fig:affine_planes}}
{}
\end{figure}

The affine model also yields two exact quantitative laws, both confirmed
sharply by the data. First, the scaling law
$\mathcal E^{\UE}=2\mathcal E^{\SO}$: \cref{fig:scaling_law} displays the
breakpoint spectra on the demand ray, and every one of the $247$
system-optimum breakpoints at $t$ has a matching user-equilibrium
breakpoint at $2t$. Second, the boundary maximum
principle: \cref{fig:boundary} shows that the maximum of $\PoA$ over
the closure of an active-set cell is attained in the \emph{relative
interior of a boundary edge} of the cell.

\begin{figure}[!t]
\FIGURE
{\includegraphics[width=0.62\linewidth]{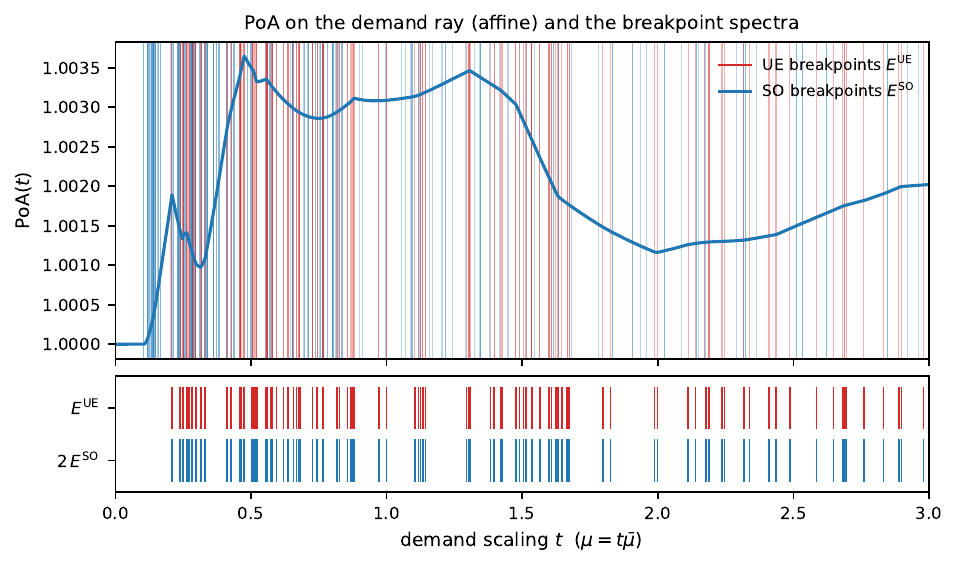}}
{Breakpoints on the demand ray (affine costs). Top:
  $\PoA(t)$ with the kinks coloured by origin (red: UE breakpoints;
  blue: SO breakpoints). Bottom: the two spectra on $[0,3]$ coincide
  exactly---every UE breakpoint sits at $2\times$ an SO breakpoint.
  \label{fig:scaling_law}}
{}
\end{figure}

\begin{figure}[!t]
\FIGURE
{\includegraphics[width=0.62\linewidth]{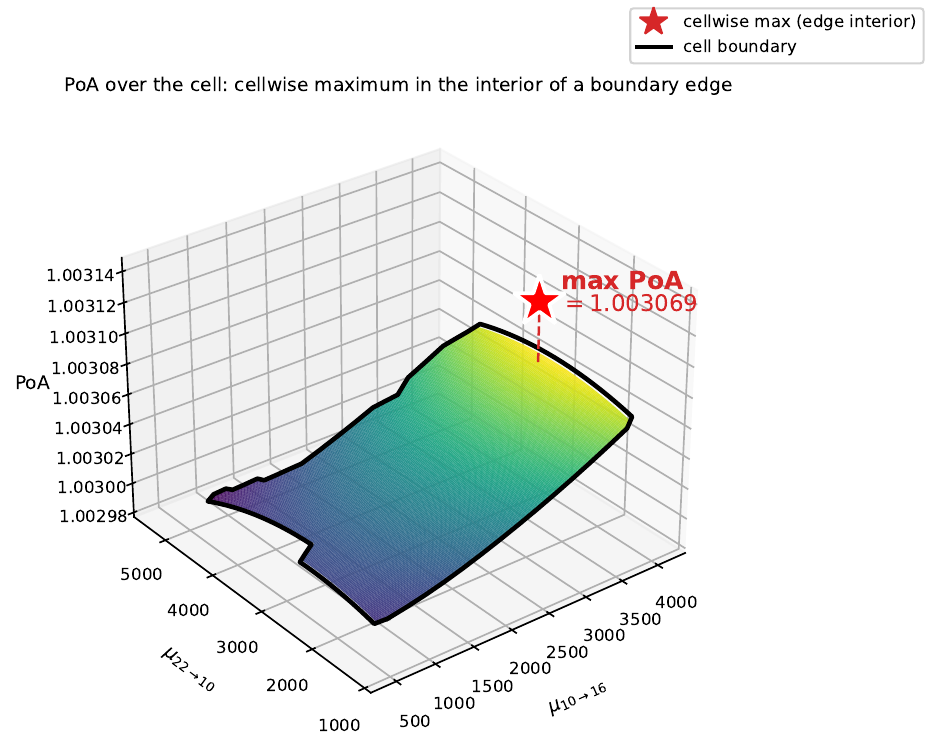}}
{The $\PoA$ surface over a single active-set cell (affine
  costs): the cellwise maximum (red star, $\PoA=1.003069$) is attained
  in the relative interior of a boundary edge of the cell, in
  accordance with the boundary maximum principle.
  \label{fig:boundary}}
{}
\end{figure}

\subsection{Local structure of the PoA near kinks}
\label{sec:num-kink}

We close with two neighbouring nondifferentiable points examined at
full resolution on the BPR demand line that varies $\mu_{10\to16}$
within the two-dimensional slice. \cref{fig:feature} (top) shows the
local $\PoA$ with corners at $\mu_{10\to16}\approx6421$ and
$\approx6654$. The bottom row displays, at one sample point in each of
the three segments separated by the two kinks, the active paths of the
two OD pairs whose active sets change at the kinks, together with the
routing of the varied OD pair itself. Crossing the first kink, the
active path $4{\to}5{\to}9{\to}10{\to}16$ of the OD pair $4{\to}16$
leaves the network (purple; present at $6321$, gone at $6521$), while
its other two active paths persist throughout. The second kink at
$\approx6654$ is generated by yet another OD pair: $10{\to}7$ gains the
additional active path $10{\to}9{\to}8{\to}7$ (red) alongside its
original path $10{\to}16{\to}18{\to}7$ (orange). The varied OD pair
$10{\to}16$ itself keeps a single active path---the direct
link---across the whole window (grey). A kink of the value of $\PoA$
along a demand line therefore need not be caused by the OD pair whose
demand is varied.

\begin{figure}[!t]
\FIGURE
{\includegraphics[width=0.85\linewidth]{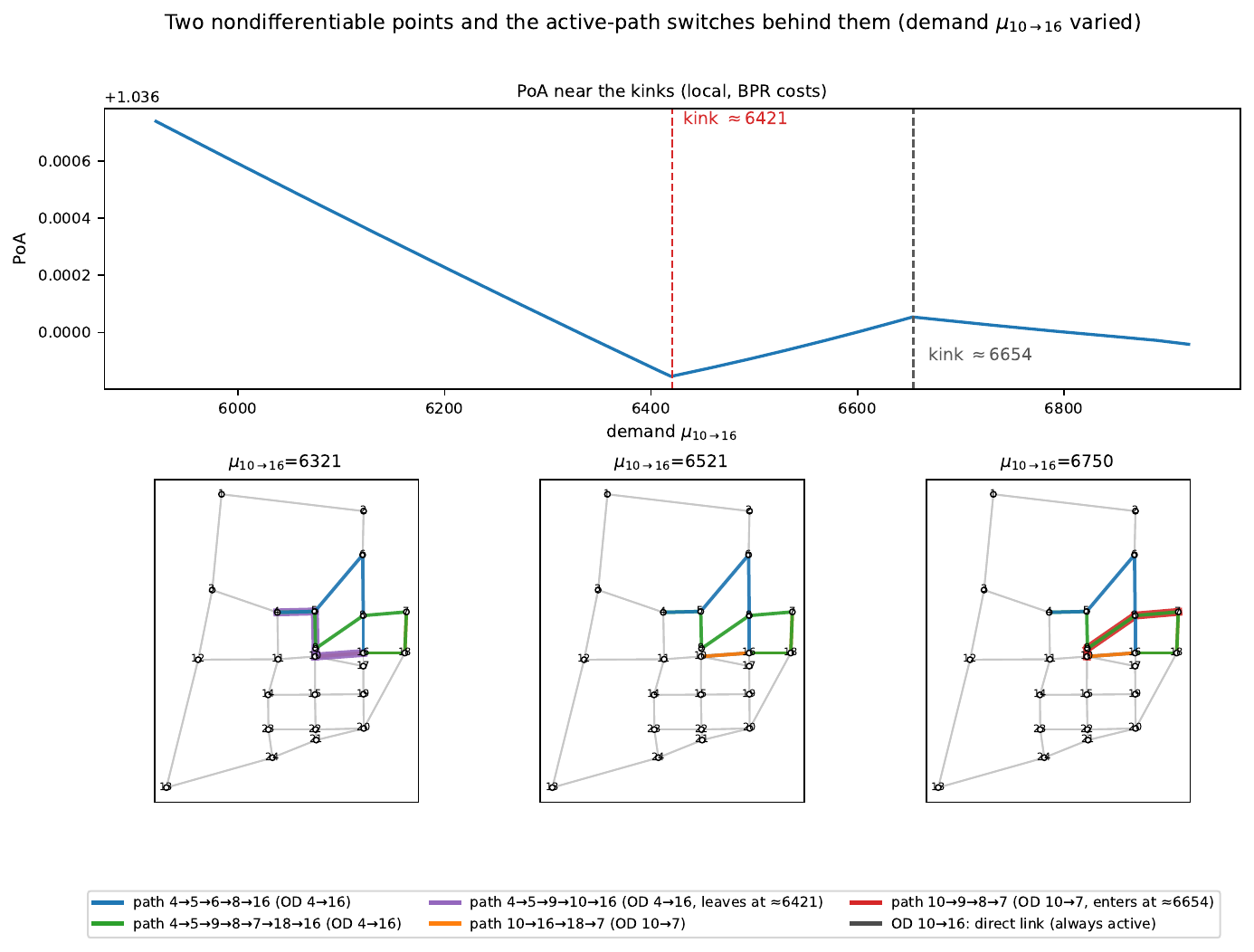}}
{Top: local $\PoA$ near the two kinks on the demand line
  $\mu_{10\to16}$. Bottom: the active paths in the three segments
  separated by the two kinks; cool colours for the OD pair $4{\to}16$
  (the purple path leaves at $\approx6421$), warm colours for
  $10{\to}7$ (the red path enters at $\approx6654$), and grey for the
  varied OD pair $10{\to}16$ (direct link, always active).
  \label{fig:feature}}
{}
\end{figure}

\section{Conclusion and Future Works}
\label{sec:concluding}

We have developed a regularity theory of the price of anarchy as a function
of the full demand vector: a multivariate \(C^1\) framework on active-set
interiors; a geometric characterization of the nondifferentiability set;
sharp conditions for higher-order regularity; and an affine specialization
with an exact scaling law and a boundary maximum principle for the PoA. We
conclude with two remarks.

\paragraph{Relation to the differentiability question of Wu and M\"ohring.}
In \citeauthor{wu_mohring2023}~\cite{wu_mohring2023}, the PoA is viewed as a map on a metric space of
games, and the authors ask under which conditions the upper and lower
derivatives of this map coincide \cite[Open~Question~4.2]{wu_mohring2023}.
Noting that the topology of the full game space is too complicated for a
direct treatment, they suggested considering parametric demands in a
finite-dimensional Euclidean space as a promising first step, and explicitly
left the resulting differentiability analysis of the PoA for future work. The
present paper carries out exactly this program on the demand slice of their
game space: by \cref{C1_regularity_of_PoA} the PoA is differentiable at every
point of the regular region \(\mathcal O\); by
\cref{prop:nondiff-measure-zero} the residual set has measure zero; and under
the strictly-increasing/constant cost structure, the candidates for failure
are reduced to a countable union of codimension-one \(C^1\) submanifolds
(\cref{thm:nondiff-countable-manifold-cover}), which need not materialize
even when active sets switch on a set of positive measure
(\cref{ex:E_positive_measure}). It is also worth noting that the cost
structure under which \citeauthor{wu_mohring2023}~\cite{wu_mohring2023} obtained the H\"older exponent
one---edge costs that are constant or have a strictly positive
derivative---coincides with \cref{ass:positive-constant-costs} of the present
paper. Our results give a geometric explanation for this coincidence: under
the same structure, the PoA is locally selected from finitely many \(C^1\)
branches that agree in value at their interfaces, so it is locally Lipschitz
everywhere on \(U^\circ\), which is precisely the H\"older exponent one.

\paragraph{Further directions.}
It is natural to ask when the countable manifold cover of
\cref{thm:nondiff-countable-manifold-cover} is in fact finite, and whether
the manifolds form a stratification of the demand space. The polyhedral cell
decomposition of \cref{sec:affine} suggests algorithmic approaches to the
exact piecewise computation of the PoA over the whole demand cone, extending
the one-dimensional parametric computations of
\citeauthor{klimm_warode2019}~\cite{klimm_warode2019,klimm_warode2022}. Finally, the boundary maximum principle of \cref{thm:poa-cell-boundary-maximum} reduces cellwise maxima of the PoA to boundary faces; a face-stationary characterization of such maxima would complete the picture. We leave these questions for future work.

\appendix

\section{Proofs for Section \ref{sec:model}}
\label{app:sec2-proofs}

\begin{proof}{Proof of \cref{lem:value-C1}}
	First, \(V\) is convex on \(U^\circ\). Indeed, let
	\(\bm\mu,\bm\nu\in U^\circ\), let \(f^\mu\in\mathcal F(\bm\mu)\) and
	\(f^\nu\in\mathcal F(\bm\nu)\) be Beckmann minimizers, and let
	\(\theta\in[0,1]\). Since
	\(R\bigl(\theta f^\mu+(1-\theta)f^\nu\bigr)=\theta\bm\mu+(1-\theta)\bm\nu\),
	the flow \(\theta f^\mu+(1-\theta)f^\nu\) is feasible for demand
	\(\theta\bm\mu+(1-\theta)\bm\nu\). By convexity of \(B\),
	\[
	\begin{aligned}
		V(\theta\bm\mu+(1-\theta)\bm\nu)
		&\le
		B\bigl(\theta f^\mu+(1-\theta)f^\nu\bigr) \\
		&\le
		\theta B(f^\mu)+(1-\theta)B(f^\nu)  \\
		&=
		\theta V(\bm\mu)+(1-\theta)V(\bm\nu).
	\end{aligned}
	\]
	Thus \(V\) is a finite convex function on the open convex set \(U^\circ\).
	
	Fix \(\bm\mu\in U^\circ\), and let \(f^*\in\argmin_{f\in\mathcal F(\bm\mu)}B(f)\).
	By the Beckmann--Wardrop characterization, there exists the equilibrium cost vector
	\(\bm\lambda(\bm\mu)=(\lambda_1(\bm\mu),\ldots,\lambda_n(\bm\mu))^\top\)
	such that, for every \(p\in\mathcal P_i\), \(c_p(f^*)\ge \lambda_i(\bm\mu),\)
	with equality whenever \(f_p^*>0\).

	We show that \(\bm\lambda(\bm\mu)\in\partial V(\bm\mu)\). Let
	\(\bm\nu\in U^\circ\), and take any \(g\in\mathcal F(\bm\nu)\). Since \(B\) is convex and
	differentiable, \(B(g)\ge B(f^*)+\nabla B(f^*)^\top(g-f^*)\).
	For each path \(p\in\mathcal P_i\), one has \((\nabla B(f^*))_p=c_p(f^*).\)
	Hence
	\[
	\sum_{p\in\mathcal P}c_p(f^*)g_p
	=
	\sum_{i=1}^n\sum_{p\in\mathcal P_i}c_p(f^*)g_p
	\ge
	\sum_{i=1}^n\lambda_i(\bm\mu)\sum_{p\in\mathcal P_i}g_p
	=
	\bm\lambda(\bm\mu)^\top\bm\nu.
	\]
	Similarly, because \(f_p^*>0\) only on paths whose cost is equal to the corresponding
	equilibrium OD cost,
	\[
	\sum_{p\in\mathcal P}c_p(f^*)f_p^*
	=
	\sum_{i=1}^n\lambda_i(\bm\mu)\sum_{p\in\mathcal P_i}f_p^*
	=
	\bm\lambda(\bm\mu)^\top\bm\mu.
	\]
	Therefore
	\[
	B(g)\ge V(\bm\mu)+\bm\lambda(\bm\mu)^\top(\bm\nu-\bm\mu).
	\]
	Taking the infimum over \(g\in\mathcal F(\bm\nu)\) gives
	\(V(\bm\nu)\ge V(\bm\mu)+\bm\lambda(\bm\mu)^\top(\bm\nu-\bm\mu)\),
	so \(\bm\lambda(\bm\mu)\in\partial V(\bm\mu)\).
	
	It remains to show that this subgradient is unique. Let
	\(\bm\xi\in\partial V(\bm\mu)\). By the standard perturbation-duality
	characterization of the value function
	\(V(\bm\mu)=\inf\{B(f): f\ge0,\ Rf=\bm\mu\}\)
	(see, e.g., Theorems~28.2 and~29.1 of \citeauthor{rockafellar1970}~\cite{rockafellar1970}; the constraints are
	linear, and \(V\) is finite on the open set \(U^\circ\)),
	the vector \(\bm\xi\) is an optimal multiplier for the equality constraint
	\(Rf=\bm\mu\). Equivalently, for some Beckmann minimizer \(f^*\), the
	Karush--Kuhn--Tucker conditions give, for every \(p\in\mathcal P_i\),
	\(c_p(f^*)\ge \xi_i\), and \(c_p(f^*)=\xi_i\) whenever \(f_p^*>0\).
	Thus \(\bm\xi\) is an OD equilibrium cost vector associated with a Wardrop
	equilibrium. Since the equilibrium edge-cost vector is unique (see the discussion
	and references in \cref{sec:model}) and hence the induced
	OD equilibrium cost vector is unique, one must have \(\bm\xi=\bm\lambda(\bm\mu).\)
	Consequently, \(\partial V(\bm\mu)=\{\bm\lambda(\bm\mu)\}\).
	Since this holds for every \(\bm\mu\in U^\circ\), the finite convex function \(V\) is
	differentiable everywhere on \(U^\circ\), with
	\(D_{\bm\mu}V(\bm\mu)=\bm\lambda(\bm\mu)\).
	Here we use the facts that a convex function is differentiable at a point if and
	only if its subdifferential there is a singleton, and that a differentiable convex
	function on an open convex set has a continuous gradient (Theorems~25.1 and~25.5
	of \citeauthor{rockafellar1970}~\cite{rockafellar1970}, respectively). Hence
	\(V\in C^1(U^\circ)\) and \(D_{\bm\mu}V(\bm\mu)=\bm\lambda(\bm\mu)\),
	and \(\bm\lambda\) is continuous on \(U^\circ\).\Halmos
\end{proof}

\begin{proof}{Proof of \cref{thm:local-C1-CUE}}
	Since \(\bm\mu^0\in \operatorname{int}(\Omega_\alpha)\),
	there exists an open neighborhood \(N\subset U^\circ\) of \(\bm\mu^0\) such that \(N\subset \Omega_\alpha .\)
	Hence, for every \(\bm\mu\in N\), the original Beckmann problem admits a minimizer
	whose positive-flow support is exactly \(\alpha\).
	
	For such a demand vector \(\bm\mu\), let \(f^{*}\) be a Beckmann minimizer of the
	original problem with support \(\alpha\). By deleting the zero components outside
	\(\alpha\), we obtain a feasible solution of the restricted Beckmann problem associated
	with \(\alpha\). Therefore, \(V_\alpha(\bm\mu)\le V(\bm\mu).\)
	Conversely, every feasible flow of the restricted problem is also feasible for the
	original Beckmann problem after inserting zero components on the paths outside
	\(\alpha\). Hence \(V(\bm\mu)\le V_\alpha(\bm\mu).\)
	Thus \(V(\bm\mu)=V_\alpha(\bm\mu)\) for every \(\bm\mu\in N\).
	
	We next recall the value-function regularity of the restricted problem. For fixed
	\(\bm\mu\in N\), introduce the perturbation function
	\[
	\varphi^\alpha_{\bm\mu}(f_\alpha,z)
	=
	\begin{cases}
		B_\alpha(f_\alpha),
		&
		f_\alpha\ge 0,\quad R_\alpha f_\alpha=\bm\mu+z,
		\\[0.5ex]
		+\infty,
		&
		\text{otherwise}.
	\end{cases}
	\]
	Define \(v^\alpha_{\bm\mu}(z):=\inf_{f_\alpha}\varphi^\alpha_{\bm\mu}(f_\alpha,z)\).
	Then, by construction, \(v^\alpha_{\bm\mu}(z)=V_\alpha(\bm\mu+z)\).
	Since \(\varphi^\alpha_{\bm\mu}\) is proper, closed, and convex, the map \(z\longmapsto V_\alpha(\bm\mu+z)\)
	is convex. Hence \(V_\alpha\) is convex on \(N\).
	
	We now identify its subdifferential by convex duality. The Fenchel conjugate of
	\(\varphi^\alpha_{\bm\mu}\) at \((0,\lambda)\) is
	\[
	(\varphi^\alpha_{\bm\mu})^*(0,\lambda)
	=
	\sup_{f_\alpha,z}
	\left\{
	\langle \lambda,z\rangle
	-
	\varphi^\alpha_{\bm\mu}(f_\alpha,z)
	\right\}.
	\]
	Using the constraint \(R_\alpha f_\alpha=\bm\mu+z,\)
	we obtain
	\[
	(\varphi^\alpha_{\bm\mu})^*(0,\lambda)
	=
	\sup_{f_\alpha\ge0}
	\left\{
	\langle \lambda,R_\alpha f_\alpha-\bm\mu\rangle
	-
	B_\alpha(f_\alpha)
	\right\}.
	\]
	Thus the corresponding dual problem is \(\min_{\lambda\in\mathbb R^n}(\varphi^\alpha_{\bm\mu})^*(0,\lambda)\).
	
	Since the restricted feasible set is nonempty on \(N\), the standard convex-duality
	description of the subdifferential of a value function (see again Theorems~28.2
	and~29.1 of \citeauthor{rockafellar1970}~\cite{rockafellar1970}) gives
	\[
	\partial V_\alpha(\bm\mu)=S(D^\alpha_{\bm\mu}),
	\]
	where \(S(D^\alpha_{\bm\mu})\) denotes the set of optimal solutions of the above dual
	problem.

	Let \(f_\alpha^*(\bm\mu)\) be a restricted Beckmann minimizer. If \(\lambda\) is dual
	optimal, then \(f_\alpha^*(\bm\mu)\) maximizes
	\(\langle \lambda,R_\alpha f_\alpha\rangle-B_\alpha(f_\alpha)\)
	over \(f_\alpha\ge0\). Therefore, the first-order optimality conditions give, for each
	retained path \(p\in \alpha\cap\mathcal P_i\),
	\(\lambda_i-c_p(f_\alpha^*(\bm\mu))=0\) if \(f_p^*(\bm\mu)>0\),
	and \(\lambda_i-c_p(f_\alpha^*(\bm\mu))\le0\) if \(f_p^*(\bm\mu)=0\).
	These are precisely the Wardrop conditions \eqref{eq:wardrop-equal}--\eqref{eq:wardrop-ineq} for the restricted network. Hence the dual
	optimizer is the restricted equilibrium-cost vector \(\lambda^\alpha(\bm\mu).\)
	Consequently, \(\partial V_\alpha(\bm\mu)=\{\lambda^\alpha(\bm\mu)\}\).
	Thus \(V_\alpha\) is differentiable at \(\bm\mu\), and
	\(D_{\bm\mu}V_\alpha(\bm\mu)=\lambda^\alpha(\bm\mu)\).
	Since this holds for every \(\bm\mu\in N\), and \(V_\alpha\) is convex, its gradient is
	continuous on \(N\). Therefore
	\(V_\alpha\in C^1(N)\) and \(D_{\bm\mu}V_\alpha(\bm\mu)=\lambda^\alpha(\bm\mu)\).

	Combining this with the local identity
	\(V(\bm\mu)=V_\alpha(\bm\mu)\) for \(\bm\mu\in N\),
	we obtain
	\[
	D_{\bm\mu}V(\bm\mu)=D_{\bm\mu}V_\alpha(\bm\mu)=\lambda^\alpha(\bm\mu),
	\qquad \bm\mu\in N.
	\]
	On the other hand, by the global Beckmann value-function identity of
	\cref{lem:value-C1},
	\(D_{\bm\mu}V(\bm\mu)=\lambda(\bm\mu)\).
	Hence \(\lambda(\bm\mu)=\lambda^\alpha(\bm\mu)\) for every \(\bm\mu\in N\).
	In particular, \(\lambda\) is continuous on \(N\).

	Finally, by definition of the user-equilibrium total cost,
	\(C^{\UE}(\bm\mu)=\bm\mu^\top \lambda(\bm\mu)\).
	Using the local identity \(\lambda=\lambda^\alpha\) and the \(C^1\) regularity of the
	restricted value function \(V_\alpha\), this gives the local \(C^1\) regularity of the
	user-equilibrium cost on the active-set interior: \(C^{\UE}\in C^1(N).\)
	Therefore there exists an open neighborhood \(N\subset U^\circ\) of \(\bm\mu^0\) such
	that \(C^{\UE}\in C^1(N).\)\Halmos
\end{proof}

\section{Proofs for Section \ref{sec:nondiff}}
\label{app:sec3-proofs}

To prove \cref{lem:local-extension-separation}, we first record the following local extension result.

\begin{lemma}[Local \(C^1\) extension branches]
	\label{lem:proof-local-extension-branches}
	Suppose that \cref{ass:SO_Cost_positive,ass:positive-constant-costs} hold, and that
	\(xc_e(x)\) is convex for every \(e\). Fix \(\bm\mu^0\in U^\circ\), and let
	\(\mathcal A(\bm\mu^0):=\{\alpha\subseteq\mathcal P:\ \bm\mu^0\in\operatorname{cl}_{U^\circ}(\Omega_\alpha)\}\).
	Then there exists an open neighborhood \(W\subset U^\circ\) of \(\bm\mu^0\) such
	that:
	\begin{enumerate}
		\item for every \(\alpha\in\mathcal A(\bm\mu^0)\), there exists a \(C^1\) map \(\widehat P_\alpha:W\to\mathbb R\)
		satisfying \(\widehat P_\alpha(\bm\mu)=\PoA(\bm\mu)\) for all \(\bm\mu\in W\cap\Omega_\alpha\);
		
		\item if \(\gamma\notin\mathcal A(\bm\mu^0)\), then, after shrinking \(W\) if
		necessary, \(W\cap\Omega_\gamma=\varnothing.\)
	\end{enumerate}
\end{lemma}

\begin{proof}{Proof}
	We first record a value-level uniqueness fact that will be used repeatedly. Fix a
	demand \(\bm\mu\), and let \(f\) and \(g\) be two Beckmann minimizers at this demand.
	By \cref{ass:positive-constant-costs}, every edge belongs either to
	\(E_+:=\{e\in E:\ c'_e(t)>0\ \text{for all }t\ge0\}\),
	or to
	\(E_0:=\{e\in E:\ c_e \text{ is constant}\}\).
	For \(e\in E_+\), the primitive \(x\mapsto \int_0^x c_e(t)\,dt\)
	is strictly convex, while for \(e\in E_0\) it is affine. If \(C_+f\ne C_+g,\)
	where \(C_+\) denotes the rows of \(C\) indexed by \(E_+\), then the midpoint \(\frac{f+g}{2}\)
	is feasible for the same demand and gives a strictly smaller Beckmann value than \(\frac{B(f)+B(g)}{2},\)
	contradicting the optimality of \(f\) and \(g\). Therefore all Beckmann minimizers
	at the same demand induce the same positive-slope edge-load vector: \(C_+f=C_+g.\)
	In particular, since the costs on \(E_0\) are constant, the induced edge-cost vector
	and the induced OD equilibrium-cost vector are value-level unique.
	
	Now fix \(\alpha\in\mathcal A(\bm\mu^0).\)
	If \(\alpha\) is not admissible, namely if \(\alpha_i=\varnothing\) for some OD pair
	\(i\), then \(W\cap\Omega_\alpha=\varnothing\) for every sufficiently small
	neighborhood \(W\) of \(\bm\mu^0\). Such an \(\alpha\) does not contribute to the
	local representation. Hence we assume below that \(\alpha\) is admissible.
	
	After deleting all paths outside \(\alpha\), write the restricted path-flow vector as \(f_\alpha\in\mathbb R^{|\alpha|},\)
	and denote the corresponding incidence matrices by \(C_\alpha, R_\alpha\).
	Since \(\alpha_i\neq\varnothing\) for every \(i\), the matrix \(R_\alpha\) has full
	row rank. Choose a linear parametrization of the affine solution space \(R_\alpha f_\alpha=\bm\mu\)
	in the form
	\(f_\alpha=T_\alpha^+y+T_\alpha^\mu\bm\mu\), \(y\in\mathbb R^{|\alpha|-n}\),
	where \(\operatorname{Im}(T_\alpha^+)=\ker R_\alpha\).
	This parametrization is used only to construct local extensions; along the extension
	we do not impose \(f_\alpha\ge0\).
	
	Since \(\alpha\in\mathcal A(\bm\mu^0),\)
	there exists a sequence
	\(\bm\mu^k\in\Omega_\alpha\) with \(\bm\mu^k\to\bm\mu^0\).
	For each \(k\), choose a Beckmann minimizer at demand \(\bm\mu^k\) whose positive
	support is exactly \(\alpha\), delete all paths outside \(\alpha\), and write the
	resulting restricted flow as
	\(f_\alpha^k=T_\alpha^+y_\alpha^k+T_\alpha^\mu\bm\mu^k\).
	Since the demands remain in a bounded neighborhood of \(\bm\mu^0\), the corresponding
	feasible flow sets are uniformly bounded. Because \(T_\alpha^+\) has full column
	rank, the uniform boundedness of the flows \(f_\alpha^k\) yields the boundedness of
	the vectors \(y_\alpha^k\). Passing to a subsequence if necessary, we
	may assume that \(y_\alpha^k\to y_\alpha^0.\)
	Set \(f_\alpha^0:=T_\alpha^+y_\alpha^0+T_\alpha^\mu\bm\mu^0\) and \(x_\alpha^0:=C_\alpha f_\alpha^0\).
	By continuity of \(B\) and of the Beckmann value function, \(f_\alpha^0\), after
	zero extension outside \(\alpha\), is a Beckmann minimizer at demand \(\bm\mu^0\).
	
	For each \(e\in E_+\), choose an open interval \(I_e\subset\mathbb R\) containing
	\(x_{\alpha,e}^0\) and a \(C^1\) local extension \(\widehat c_e:I_e\to\mathbb R\)
	such that \(\widehat c_e(x)=c_e(x)\) whenever \(x\in I_e\cap[0,+\infty)\),
	and \(\widehat c'_e(x)>0\) for all \(x\in I_e\).
	This is possible because \(c'_e(x_{\alpha,e}^0)>0\). If \(x_{\alpha,e}^0=0\), take a
	one-sided linear \(C^1\) extension on a sufficiently small interval to the left of
	zero. For \(e\in E_0\), extend \(c_e\) locally as the same constant function.
	
	Let \(C_{\alpha,+}\)
	denote the rows of \(C_\alpha\) indexed by \(E_+\), and let \(C_{\alpha,0}\) denote
	the rows indexed by \(E_0\). By shrinking the neighborhood of \((y_\alpha^0,\bm\mu^0)\)
	if necessary, we may assume that all \(E_+\)-edge loads generated in this
	neighborhood belong to the intervals \(I_e\). On this local neighborhood, define
	\(\widehat x_\alpha(y,\bm\mu):=C_\alpha T_\alpha^+y+C_\alpha T_\alpha^\mu\bm\mu\)
	and
	\(\widehat B_\alpha(y,\bm\mu):=\sum_{e\in E}\int_0^{(\widehat x_\alpha(y,\bm\mu))_e}\widehat c_e(s)\,ds\).
	The first-order equation in the free variable \(y\) is
	\[
	\widehat L_\alpha(y,\bm\mu)
	:=
	D_y\widehat B_\alpha(y,\bm\mu)
	=
	(T_\alpha^+)^\top C_\alpha^\top
	\widehat{\bm c}\bigl(\widehat x_\alpha(y,\bm\mu)\bigr).
	\]
	Because \(f_\alpha^k>0\) on all retained paths, \(f_\alpha^k\) is an interior
	minimizer of the restricted Beckmann problem. Hence
	\(\widehat L_\alpha(y_\alpha^k,\bm\mu^k)=0\).
	Passing to the limit gives
	\(\widehat L_\alpha(y_\alpha^0,\bm\mu^0)=0\).

	Set \(K_\alpha:=\ker(C_{\alpha,+}T_\alpha^+)\) and \(M_\alpha:=K_\alpha^\perp\).
	Let \(U_\alpha\) be a matrix whose columns form an orthonormal basis of
	\(M_\alpha\). Decompose \(y_\alpha^0=U_\alpha z_\alpha^0+w_\alpha^0\) with \(w_\alpha^0\in K_\alpha\).
	Fix the \(K_\alpha\)-component \(w_\alpha^0\), and define
	\[
	G_\alpha(z,\bm\mu)
	:=
	U_\alpha^\top
	\widehat L_\alpha(U_\alpha z+w_\alpha^0,\bm\mu).
	\]
	
	For every \(v\in K_\alpha\), \(C_{\alpha,+}T_\alpha^+v=0.\)
	Hence the positive-slope edge loads do not vary in the direction \(v\). Moreover,
	for \(e\in E_0\), the extended edge cost is constant. Therefore
	\[
	v^\top \widehat L_\alpha(y,\bm\mu)
	=
	\bigl(C_\alpha T_\alpha^+v\bigr)^\top
	\widehat{\bm c}\bigl(\widehat x_\alpha(y,\bm\mu)\bigr)
	=
	\bigl(C_{\alpha,0}T_\alpha^+v\bigr)^\top\bm\kappa_0,
	\]
	where \(\bm\kappa_0\) is the constant edge-cost vector on \(E_0\). Thus
	\(v^\top \widehat L_\alpha(y,\bm\mu)\) is independent of \((y,\bm\mu)\). Since
	\(\widehat L_\alpha(y_\alpha^0,\bm\mu^0)=0\),
	we obtain \(v^\top \widehat L_\alpha(y,\bm\mu)=0\) for all \(v\in K_\alpha\).
	Hence \(\widehat L_\alpha(y,\bm\mu)\in M_\alpha\),
	and consequently, locally around \((z_\alpha^0,\bm\mu^0)\),
	\(G_\alpha(z,\bm\mu)=0\) if and only if \(\widehat L_\alpha(U_\alpha z+w_\alpha^0,\bm\mu)=0\).

	The derivative of \(G_\alpha\) with respect to \(z\) is
	\[
	D_zG_\alpha(z,\bm\mu)
	=
	U_\alpha^\top
	(T_\alpha^+)^\top
	C_{\alpha,+}^\top
	D_+(z,\bm\mu)
	C_{\alpha,+}T_\alpha^+
	U_\alpha,
	\]
	where
	\(D_+(z,\bm\mu):=\operatorname{diag}\bigl(\widehat c'_e\bigl((\widehat x_\alpha(U_\alpha z+w_\alpha^0,\bm\mu))_e\bigr)\bigr)_{e\in E_+}\).
	For every \(\xi\ne0\),
	\[
	\xi^\top D_zG_\alpha(z_\alpha^0,\bm\mu^0)\xi
	=
	\bigl(C_{\alpha,+}T_\alpha^+U_\alpha\xi\bigr)^\top
	D_+(z_\alpha^0,\bm\mu^0)
	\bigl(C_{\alpha,+}T_\alpha^+U_\alpha\xi\bigr).
	\]
	Since \(U_\alpha\xi\in K_\alpha^\perp\) and \(U_\alpha\xi\ne0\),
	\(C_{\alpha,+}T_\alpha^+U_\alpha\xi\ne0\).
	Moreover \(D_+(z_\alpha^0,\bm\mu^0)\) is positive definite. Therefore
	\(\xi^\top D_zG_\alpha(z_\alpha^0,\bm\mu^0)\xi>0\).
	Hence \(D_zG_\alpha(z_\alpha^0,\bm\mu^0)\) is nonsingular.

	By the implicit function theorem, there exist a neighborhood \(W_\alpha\subset U^\circ\)
	of \(\bm\mu^0\) and a \(C^1\) map
	\(z_\alpha^*:W_\alpha\to\mathbb R^{\dim M_\alpha}\)
	such that \(z_\alpha^*(\bm\mu^0)=z_\alpha^0\) and \(G_\alpha(z_\alpha^*(\bm\mu),\bm\mu)=0\) for all \(\bm\mu\in W_\alpha\).
	Define \(\widehat y_\alpha^*(\bm\mu):=U_\alpha z_\alpha^*(\bm\mu)+w_\alpha^0\).
	Then \(\widehat L_\alpha(\widehat y_\alpha^*(\bm\mu),\bm\mu)=0\) for all \(\bm\mu\in W_\alpha\).

	The induced extended edge-load and edge-cost vectors are
	\(\widehat x_\alpha^*(\bm\mu):=\widehat x_\alpha(\widehat y_\alpha^*(\bm\mu),\bm\mu)\)
	and \(\widehat\tau_\alpha(\bm\mu):=\widehat{\bm c}(\widehat x_\alpha^*(\bm\mu))\).
	Both are \(C^1\) on \(W_\alpha\).
	
	We now define the corresponding extended OD-cost vector. From
	\(\widehat L_\alpha(\widehat y_\alpha^*(\bm\mu),\bm\mu)=0\)
	we have
	\((T_\alpha^+)^\top C_\alpha^\top\widehat\tau_\alpha(\bm\mu)=0\).
	Since \(\operatorname{Im}(T_\alpha^+)=\ker R_\alpha\),
	this gives \(C_\alpha^\top\widehat\tau_\alpha(\bm\mu)\in(\ker R_\alpha)^\perp\).
	Because \(R_\alpha\) has full row rank,
	\((\ker R_\alpha)^\perp=\operatorname{Range}(R_\alpha^\top)\).
	Hence there exists a unique vector
	\(\widehat{\bm\lambda}_\alpha(\bm\mu)\in\mathbb R^n\)
	such that
	\(C_\alpha^\top\widehat\tau_\alpha(\bm\mu)=R_\alpha^\top\widehat{\bm\lambda}_\alpha(\bm\mu)\).
	Equivalently, for every OD pair \(i\) and every retained path \(p\in\alpha_i\),
	\(\sum_{e\in p}\widehat\tau_{\alpha,e}(\bm\mu)=\widehat\lambda_{\alpha,i}(\bm\mu)\).
	Thus the value of \(\widehat\lambda_{\alpha,i}\) is independent of the choice of
	\(p\in\alpha_i\). Since \(R_\alpha\) has full row rank,
	\(\widehat{\bm\lambda}_\alpha\) is obtained linearly from
	\(C_\alpha^\top\widehat\tau_\alpha\). Therefore
	\(\widehat{\bm\lambda}_\alpha\in C^1(W_\alpha;\mathbb R^n)\).

	Set \(\widehat C_\alpha^{\UE}(\bm\mu):=\bm\mu^\top\widehat{\bm\lambda}_\alpha(\bm\mu)\).
	Then \(\widehat C_\alpha^{\UE}\in C^1(W_\alpha)\).
	By \cref{prop:BMW_for_SOcost}, the SO--Beckmann reduction, together
	with the Beckmann value-function regularity of \cref{lem:value-C1}, gives \(C^{\SO}\in C^1(U^\circ).\)
	Moreover, by \cref{ass:SO_Cost_positive},
	\(C^{\SO}(\bm\mu)>0\) for all \(\bm\mu\in U^\circ\).
	Define
	\(\widehat P_\alpha(\bm\mu):=\widehat C_\alpha^{\UE}(\bm\mu)/C^{\SO}(\bm\mu)\) for \(\bm\mu\in W_\alpha\).
	By the assumed SO value-function regularity, \(\widehat P_\alpha\in C^1(W_\alpha).\)
	
	It remains to identify this extension with the true \(\PoA\) on
	\(W_\alpha\cap\Omega_\alpha\). We first show that, after shrinking \(W_\alpha\) if
	necessary, every Beckmann minimizer with positive support \(\alpha\) at a demand \(\bm\mu\in W_\alpha\cap\Omega_\alpha\)
	has its positive-slope edge-load vector in the local extension neighborhood.
	
	Indeed, otherwise there would exist
	\(\bm\mu^j\in\Omega_\alpha\) with \(\bm\mu^j\to\bm\mu^0\),
	and restricted Beckmann minimizers \(\bar f_\alpha^j\) with support \(\alpha\) such
	that \(C_{\alpha,+}\bar f_\alpha^j\) stays outside that neighborhood. Passing to a
	subsequence, we may assume that \(\bar f_\alpha^j\to \bar f_\alpha^0.\)
	By continuity of \(B\) and of the Beckmann value function, \(\bar f_\alpha^0\), after
	zero extension outside \(\alpha\), is a Beckmann minimizer at demand \(\bm\mu^0\).
	By the value-level uniqueness fact proved at the beginning of this proof, all
	Beckmann minimizers at demand \(\bm\mu^0\) have the same positive-slope edge-load
	vector. Since \(f_\alpha^0\), after zero extension, is also a Beckmann minimizer at
	\(\bm\mu^0\), we get \(C_{\alpha,+}\bar f_\alpha^0=C_{\alpha,+}f_\alpha^0\).
	This contradicts the choice of \(\bar f_\alpha^j\). Hence the claim follows.

	Now take \(\bm\mu\in W_\alpha\cap\Omega_\alpha\).
	Choose a Beckmann minimizer of the original problem whose positive-flow support is
	exactly \(\alpha\). After deleting the zero components outside \(\alpha\), write this
	restricted minimizer as
	\(\bar f_\alpha=T_\alpha^+\bar y_\alpha+T_\alpha^\mu\bm\mu\).
	Since \(\bar f_\alpha>0\), it is an interior minimizer of the restricted Beckmann
	problem. By the preceding paragraph, its positive-slope edge loads lie in the local
	extension neighborhood. Hence \(\widehat L_\alpha(\bar y_\alpha,\bm\mu)=0\).
	On the other hand, \(\widehat L_\alpha(\widehat y_\alpha^*(\bm\mu),\bm\mu)=0\).
	Let \(d:=\bar y_\alpha-\widehat y_\alpha^*(\bm\mu)\).
	Subtracting the two equations and multiplying by \(d^\top\), we obtain
	\(0=d^\top\bigl[\widehat L_\alpha(\bar y_\alpha,\bm\mu)-\widehat L_\alpha(\widehat y_\alpha^*(\bm\mu),\bm\mu)\bigr]\).
	Equivalently,
	\(0=\bigl(C_\alpha T_\alpha^+d\bigr)^\top\bigl[\widehat{\bm c}\bigl(\widehat x_\alpha(\bar y_\alpha,\bm\mu)\bigr)-\widehat{\bm c}\bigl(\widehat x_\alpha(\widehat y_\alpha^*(\bm\mu),\bm\mu)\bigr)\bigr]\).
	The edges in \(E_0\) contribute zero to the difference of edge costs. Thus
	\(0=\sum_{e\in E_+}\Delta x_e\bigl[\widehat c_e(\bar x_e)-\widehat c_e(\widehat x_e)\bigr]\),
	where \(\Delta x_e:=\bar x_e-\widehat x_e.\)
	For \(e\in E_+\), both \(\bar x_e\) and \(\widehat x_e\) lie in the local interval
	on which \(\widehat c'_e>0\). Hence each summand is nonnegative and vanishes if and
	only if \(\Delta x_e=0\). Therefore \(\Delta x_e=0\) for all \(e\in E_+\).
	Thus the true restricted minimizer and the local branch induce the same loads on all
	positive-slope edges. On \(E_0\), edge costs are constant. Consequently,
	\(\widehat\tau_\alpha(\bm\mu)=\tau(\bm\mu)\),
	where \(\tau(\bm\mu)\) denotes the equilibrium edge-cost vector of the original
	network. Hence the induced OD equilibrium-cost vectors coincide:
	\(\widehat{\bm\lambda}_\alpha(\bm\mu)=\bm\lambda(\bm\mu)\).
	Therefore
	\[
	\widehat C_\alpha^{\UE}(\bm\mu)
	=
	\bm\mu^\top\widehat{\bm\lambda}_\alpha(\bm\mu)
	=
	\bm\mu^\top\bm\lambda(\bm\mu)
	=
	C^{\UE}(\bm\mu).
	\]
	It follows that
	\[
	\widehat P_\alpha(\bm\mu)
	=
	\frac{\widehat C_\alpha^{\UE}(\bm\mu)}{C^{\SO}(\bm\mu)}
	=
	\frac{C^{\UE}(\bm\mu)}{C^{\SO}(\bm\mu)}
	=
	\PoA(\bm\mu),
	\qquad
	\bm\mu\in W_\alpha\cap\Omega_\alpha.
	\]

	Finally, since \(\mathcal P\) is finite, the family of candidate active sets is
	finite. For every \(\gamma\notin\mathcal A(\bm\mu^0)\), one has
	\(\bm\mu^0\notin\operatorname{cl}_{U^\circ}(\Omega_\gamma)\).
	Hence there exists a neighborhood \(W_\gamma\) of \(\bm\mu^0\) such that
	\(W_\gamma\cap\Omega_\gamma=\varnothing\).
	Taking the intersection of the finitely many neighborhoods constructed above and
	shrinking if necessary, we obtain a common neighborhood \(W\) satisfying both
	assertions of the lemma.\Halmos
\end{proof}

\begin{proof}{Proof of \cref{lem:local-extension-separation}}
	Let \(\bm\mu^0\in\Nondiff(\PoA),\)
	and define \(\mathcal A(\bm\mu^0):=\{\alpha\subseteq\mathcal P:\ \bm\mu^0\in \operatorname{cl}_{U^\circ}(\Omega_\alpha)\}\).
	By \cref{lem:proof-local-extension-branches}, there exists a neighborhood
	\(W\subset U^\circ\) of \(\bm\mu^0\) such that, for every
	\(\alpha\in\mathcal A(\bm\mu^0)\), there is a \(C^1\) map \(\widehat P_\alpha:W\to\mathbb R\)
	satisfying \(\widehat P_\alpha(\bm\mu)=\PoA(\bm\mu)\) for all \(\bm\mu\in W\cap\Omega_\alpha\),
	and \(W\cap\Omega_\gamma=\varnothing\) for all \(\gamma\notin\mathcal A(\bm\mu^0)\).

	First, all relevant extensions agree with \(\PoA\) at \(\bm\mu^0\). Indeed, fix \(\alpha\in\mathcal A(\bm\mu^0).\)
	Then there exists a sequence
	\(\bm\mu^k\in W\cap\Omega_\alpha\) with \(\bm\mu^k\to\bm\mu^0\).
	For each \(k\), \(\widehat P_\alpha(\bm\mu^k)=\PoA(\bm\mu^k)\).
	Since \(\widehat P_\alpha\) is continuous and \(\PoA\) is continuous under the present
	assumptions, passing to the limit gives
	\(\widehat P_\alpha(\bm\mu^0)=\PoA(\bm\mu^0)\).

	Suppose, toward a contradiction, that all these local extensions have the same
	gradient at \(\bm\mu^0\). Then there exists \(a\in\mathbb R^n\) such that
	\(\nabla\widehat P_\alpha(\bm\mu^0)=a\) for all \(\alpha\in\mathcal A(\bm\mu^0)\).
	Because \(\mathcal A(\bm\mu^0)\) is finite and each \(\widehat P_\alpha\) is \(C^1\),
	the first-order expansions are uniform over \(\alpha\in\mathcal A(\bm\mu^0)\):
	\[
	\widehat P_\alpha(\bm\mu)
	=
	\widehat P_\alpha(\bm\mu^0)
	+
	a^\top(\bm\mu-\bm\mu^0)
	+
	o(\|\bm\mu-\bm\mu^0\|)
	\]
	as \(\bm\mu\to\bm\mu^0\), uniformly in \(\alpha\).

	For \(\bm\mu\in W\), choose an active set \(\alpha(\bm\mu)\) such that \(\bm\mu\in\Omega_{\alpha(\bm\mu)}.\)
	By the local exclusion property, \(\alpha(\bm\mu)\in\mathcal A(\bm\mu^0)\).
	Hence \(\PoA(\bm\mu)=\widehat P_{\alpha(\bm\mu)}(\bm\mu)\)
	for all \(\bm\mu\in W\). Using the uniform expansion and the identity
	\(\widehat P_{\alpha(\bm\mu)}(\bm\mu^0)=\PoA(\bm\mu^0)\),
	we obtain
	\[
	\PoA(\bm\mu)
	=
	\PoA(\bm\mu^0)
	+
	a^\top(\bm\mu-\bm\mu^0)
	+
	o(\|\bm\mu-\bm\mu^0\|).
	\]
	Thus \(\PoA\) is differentiable at \(\bm\mu^0\), contradicting \(\bm\mu^0\in\Nondiff(\PoA).\)
	Therefore there exist two active sets \(\alpha,\beta\in\mathcal A(\bm\mu^0)\)
	such that \(\nabla\widehat P_\alpha(\bm\mu^0)\ne\nabla\widehat P_\beta(\bm\mu^0)\).
	Together with the value identity already proved, this gives
	\(\widehat P_\alpha(\bm\mu^0)=\widehat P_\beta(\bm\mu^0)=\PoA(\bm\mu^0)\)
	and \(\nabla\widehat P_\alpha(\bm\mu^0)\ne\nabla\widehat P_\beta(\bm\mu^0)\).
\Halmos
\end{proof}

\begin{proof}{Proof of \cref{thm:nondiff-countable-manifold-cover}}
	For every point \(\bar{\bm\mu}\in\Nondiff(\PoA),\)
	apply \cref{lem:proof-local-extension-branches}. We obtain an open
	neighborhood \(W_{\bar{\bm\mu}}\subset U^\circ\), a finite family
	\(\mathcal A(\bar{\bm\mu})=\{\alpha\subseteq\mathcal P:\ \bar{\bm\mu}\in\operatorname{cl}_{U^\circ}(\Omega_\alpha)\}\),
	and \(C^1\) local branches
	\(\widehat P_\alpha^{\,\bar{\bm\mu}}:W_{\bar{\bm\mu}}\to\mathbb R\) for \(\alpha\in\mathcal A(\bar{\bm\mu})\),
	such that \(\widehat P_\alpha^{\,\bar{\bm\mu}}(\bm\mu)=\PoA(\bm\mu)\) for all \(\bm\mu\in W_{\bar{\bm\mu}}\cap\Omega_\alpha\),
	and \(W_{\bar{\bm\mu}}\cap\Omega_\gamma=\varnothing\) for all \(\gamma\notin\mathcal A(\bar{\bm\mu})\).

	We claim that \(W_{\bar{\bm\mu}}\cap\Nondiff(\PoA)\subseteq\bigcup_{\alpha,\beta\in\mathcal A(\bar{\bm\mu})}M_{\alpha\beta}^{\bar{\bm\mu}}\),
	where
	\[
	M_{\alpha\beta}^{\bar{\bm\mu}}
	:=
	\left\{
	\bm\mu\in W_{\bar{\bm\mu}}:
	\widehat P_\alpha^{\,\bar{\bm\mu}}(\bm\mu)
	=
	\widehat P_\beta^{\,\bar{\bm\mu}}(\bm\mu),
	\
	\nabla\widehat P_\alpha^{\,\bar{\bm\mu}}(\bm\mu)
	\ne
	\nabla\widehat P_\beta^{\,\bar{\bm\mu}}(\bm\mu)
	\right\}.
	\]
	Indeed, let \(\bm\nu\in W_{\bar{\bm\mu}}\cap\Nondiff(\PoA)\).
	Because \(W_{\bar{\bm\mu}}\cap\Omega_\gamma=\varnothing\) for every
	\(\gamma\notin\mathcal A(\bar{\bm\mu})\), every active-set region whose closure
	contains \(\bm\nu\) must have index in \(\mathcal A(\bar{\bm\mu})\). Thus
	\(\mathcal A(\bm\nu):=\{\alpha:\ \bm\nu\in\operatorname{cl}_{U^\circ}(\Omega_\alpha)\}\subseteq\mathcal A(\bar{\bm\mu})\).

	For every \(\alpha\in\mathcal A(\bm\nu)\), the identity
	\(\widehat P_\alpha^{\,\bar{\bm\mu}}(\bm\nu)=\PoA(\bm\nu)\)
	follows by approaching \(\bm\nu\) with points of
	\(W_{\bar{\bm\mu}}\cap\Omega_\alpha\) and passing to the limit.

	Since the family of candidate active sets is finite, and since
	\(\bm\nu\notin\operatorname{cl}_{U^\circ}(\Omega_\gamma)\)
	for every \(\gamma\notin\mathcal A(\bm\nu)\),
	there exists an open neighborhood \(N_{\bm\nu}\subset W_{\bar{\bm\mu}}\) of
	\(\bm\nu\) such that \(N_{\bm\nu}\cap\Omega_\gamma=\varnothing\) for all \(\gamma\notin\mathcal A(\bm\nu)\).
	Hence, for every \(\bm\mu\in N_{\bm\nu}\), any active-set region
	\(\Omega_{\alpha(\bm\mu)}\) containing \(\bm\mu\) must satisfy
	\(\alpha(\bm\mu)\in\mathcal A(\bm\nu)\).
	Therefore, near \(\bm\nu\), \(\PoA\) is locally selected from the finite family
	\(\{\widehat P_\alpha^{\,\bar{\bm\mu}}:\alpha\in\mathcal A(\bm\nu)\}\).

	If all gradients \(\nabla\widehat P_\alpha^{\,\bar{\bm\mu}}(\bm\nu)\), \(\alpha\in\mathcal A(\bm\nu)\),
	were equal, then the finite uniform first-order expansion of these \(C^1\) branches
	would imply that \(\PoA\) is differentiable at \(\bm\nu\), contradicting \(\bm\nu\in\Nondiff(\PoA).\)
	Therefore there exist \(\alpha,\beta\in\mathcal A(\bm\nu)\subseteq\mathcal A(\bar{\bm\mu})\)
	such that \(\widehat P_\alpha^{\,\bar{\bm\mu}}(\bm\nu)=\widehat P_\beta^{\,\bar{\bm\mu}}(\bm\nu)=\PoA(\bm\nu)\)
	and \(\nabla\widehat P_\alpha^{\,\bar{\bm\mu}}(\bm\nu)\ne\nabla\widehat P_\beta^{\,\bar{\bm\mu}}(\bm\nu)\).
	Thus \(\bm\nu\in M_{\alpha\beta}^{\bar{\bm\mu}}\),
	which proves the claim.

	For each pair \((\alpha,\beta)\), define
	\(F_{\alpha\beta}^{\bar{\bm\mu}}:=\widehat P_\alpha^{\,\bar{\bm\mu}}-\widehat P_\beta^{\,\bar{\bm\mu}}\).
	Then
	\(M_{\alpha\beta}^{\bar{\bm\mu}}=\{\bm\mu\in W_{\bar{\bm\mu}}:\ F_{\alpha\beta}^{\bar{\bm\mu}}(\bm\mu)=0,\ \nabla F_{\alpha\beta}^{\bar{\bm\mu}}(\bm\mu)\ne0\}\).
	By the implicit function theorem, every nonempty \(M_{\alpha\beta}^{\bar{\bm\mu}}\) is a local \((n-1)\)-dimensional \(C^1\)
	submanifold of \(U^\circ\).

	The neighborhoods
	\(\{W_{\bar{\bm\mu}}:\ \bar{\bm\mu}\in\Nondiff(\PoA)\}\)
	form an open cover of \(\Nondiff(\PoA)\). Since \(U^\circ\subset\mathbb R^n\) is
	second countable, this cover admits a countable subcover\(\{W_j\}_{j=1}^\infty\). 
	Each \(W_j\) is one of the neighborhoods constructed above and carries only finitely
	many local branches. Therefore the collection of all corresponding nonempty regular
	coincidence sets
	\[
	M_{j,\alpha\beta}
	:=
	\left\{
	\bm\mu\in W_j:
	\widehat P_{j,\alpha}(\bm\mu)
	=
	\widehat P_{j,\beta}(\bm\mu),
	\
	\nabla\widehat P_{j,\alpha}(\bm\mu)
	\ne
	\nabla\widehat P_{j,\beta}(\bm\mu)
	\right\}
	\]
	is at most countable.

	By the local containment claim, these sets cover all nondifferentiability points:
	\(\Nondiff(\PoA)\subseteq\bigcup_{j=1}^\infty\bigcup_{\alpha,\beta}M_{j,\alpha\beta}\).
	Relabeling this countable family as \(M_1,M_2,\ldots,\) we obtain
	\(\Nondiff(\PoA)\subseteq\bigcup_{\ell=1}^\infty M_\ell\),
	where each \(M_\ell\) is a local \((n-1)\)-dimensional \(C^1\) submanifold of
	\(U^\circ\).\Halmos
\end{proof}

\section{Proofs for Section \ref{sec:higher-order}}
\label{app:sec4-proofs}

\begin{proof}{Proof of \cref{thm:Cr_regularity_UEcost}}
	Since \(\bm\mu^0\in\operatorname{int}(\Omega_\alpha)\),
	there exists an open neighborhood \(N\subset U^\circ\) of \(\bm\mu^0\) such that \(N\subset\Omega_\alpha.\)
	Choose a Beckmann minimizer \(f_\alpha^0\) at \(\bm\mu^0\) whose positive-flow
	support is exactly \(\alpha\). Then \(f_\alpha^0>0\) and \(R_\alpha f_\alpha^0=\bm\mu^0\).

	Because every OD pair has at least one path in \(\alpha\), the matrix \(R_\alpha\) has
	full row rank. Choose matrices \(S_\alpha\) and \(T_\alpha\) such that
	\(R_\alpha S_\alpha=I_n\) and \(\operatorname{Im}(T_\alpha)=\ker R_\alpha\).
	Then every restricted feasible flow can be written as \(f_\alpha=T_\alpha y+S_\alpha\bm\mu.\)
	Let \(y^0\) be such that \(f_\alpha^0=T_\alpha y^0+S_\alpha\bm\mu^0\).

	Let \(E_+:=\{e\in E:\ c'_e(t)>0\ \text{for all }t\ge0\}\) and \(E_0:=\{e\in E:\ c_e \ \text{is constant}\}\).
	By \cref{ass:positive-constant-costs}, \(E=E_+\cup E_0\). Let
	\(C_{\alpha,+}\) denote the submatrix of \(C_\alpha\) formed by the rows indexed by
	\(E_+\).

	Define \(K:=\ker(C_{\alpha,+}T_\alpha)\) and \(M:=K^\perp\).
	Let \(U\) be a matrix whose columns form an orthonormal basis of \(M\). Decompose \(y^0=Uz^0+w^0\) with \(w^0\in K.\)

	For the restricted Beckmann objective, the first-order condition in the free variable
	\(y\) is
	\[
	L(y,\bm\mu)
	:=
	T_\alpha^\top C_\alpha^\top
	\bm c\bigl(C_\alpha(T_\alpha y+S_\alpha\bm\mu)\bigr)
	=
	0.
	\]
	Here \(\bm c(x):=(c_e(x_e))_{e\in E}.\)
	For every \(v\in K\), the positive-slope edge load variation satisfies \(C_{\alpha,+}T_\alpha v=0.\)
	The constant-cost edges contribute only constant terms to \(L\). Since \(L(y^0,\bm\mu^0)=0,\)
	it follows that \(v^\top L(y,\bm\mu)=0\) for all \(v\in K\),
	and hence \(L(y,\bm\mu)\in K^\perp=M.\)
	Therefore \(L(y,\bm\mu)=0\) is equivalent, after writing \(y=Uz+w^0\), to
	\[
	G(z,\bm\mu):=U^\top L(Uz+w^0,\bm\mu)=0.
	\]
	
	Note that \(G\) is well defined in a neighborhood of \((z^0,\bm\mu^0)\): a row of
	\(C_\alpha\) indexed by an edge \(e\) is either identically zero, in which case the
	load of \(e\) is identically zero, or nonzero, in which case
	\((C_\alpha f_\alpha^0)_e>0\) since \(f_\alpha^0>0\), so the corresponding load
	remains positive nearby. Since each \(c_e\in C^r([0,+\infty))\), the map \(G\) is
	\(C^r\) in a neighborhood of \((z^0,\bm\mu^0)\). Its derivative with respect to \(z\) is
	\[
	D_zG(z,\bm\mu)
	=
	U^\top T_\alpha^\top C_{\alpha,+}^\top
	D_+(z,\bm\mu)
	C_{\alpha,+}T_\alpha U,
	\]
	where
	\(D_+(z,\bm\mu)=\operatorname{diag}\bigl(c'_e\bigl((C_\alpha(T_\alpha(Uz+w^0)+S_\alpha\bm\mu))_e\bigr)\bigr)_{e\in E_+}\).
	By \cref{ass:positive-constant-costs}, \(D_+(z^0,\bm\mu^0)\)
	is positive definite. Thus, for every \(\xi\ne0\),
	\[
	\xi^\top D_zG(z^0,\bm\mu^0)\xi
	=
	\bigl(C_{\alpha,+}T_\alpha U\xi\bigr)^\top
	D_+(z^0,\bm\mu^0)
	\bigl(C_{\alpha,+}T_\alpha U\xi\bigr).
	\]
	Since \(U\xi\in K^\perp\) and \(U\xi\ne0\), one has \(C_{\alpha,+}T_\alpha U\xi\ne0.\)
	Therefore \(\xi^\top D_zG(z^0,\bm\mu^0)\xi>0.\)
	Hence \(D_zG(z^0,\bm\mu^0)\) is nonsingular.

	By the implicit function theorem, there exist a neighborhood
	\(V\subset N\) of \(\bm\mu^0\) and a \(C^r\) map \(z^*:V\to\mathbb R^{\dim M}\)
	such that \(z^*(\bm\mu^0)=z^0\) and \(G(z^*(\bm\mu),\bm\mu)=0\) for all \(\bm\mu\in V\).
	Define \(y^*(\bm\mu):=Uz^*(\bm\mu)+w^0,\)
	and \(f_\alpha^*(\bm\mu):=T_\alpha y^*(\bm\mu)+S_\alpha\bm\mu\).
	Then \(f_\alpha^*\in C^r(V)\) and \(R_\alpha f_\alpha^*(\bm\mu)=\bm\mu\).
	Since \(f_\alpha^0>0\), after shrinking \(V\) if necessary,
	\(f_\alpha^*(\bm\mu)>0\) for all \(\bm\mu\in V\).
	
	The equation \(L(y^*(\bm\mu),\bm\mu)=0\) is precisely the first-order condition for
	the restricted Beckmann problem on the affine space \(R_\alpha f_\alpha=\bm\mu.\)
	Because the restricted Beckmann objective is convex, this first-order condition is
	sufficient for optimality. Since \(V\subset\Omega_\alpha\), the restricted minimizer
	coincides with a Beckmann minimizer of the original problem with active path set
	\(\alpha\).
	
	The induced edge-load vector is \(x^*(\bm\mu):=C_\alpha f_\alpha^*(\bm\mu)\).
	Since \(C_\alpha\) is linear and \(f_\alpha^*\in C^r(V)\), \(x^*\in C^r(V).\)
	For every edge \(e\), \(\tau_e(\bm\mu)=c_e(x_e^*(\bm\mu))\).
	Because \(c_e\in C^r\), it follows that \(\tau_e\in C^r(V).\)
	Thus \(\bm\tau\in C^r(V).\)

	For each OD pair \(i\), choose any active path \(p_i\in\alpha_i\). Along the active
	branch, all active paths for OD pair \(i\) have equal equilibrium cost, so
	\(\lambda_i(\bm\mu)=\sum_{e\in p_i}\tau_e(\bm\mu)\).
	This is a finite sum of \(C^r\) functions; hence
	\(\lambda_i\in C^r(V)\) for \(i=1,\ldots,n\).
	Therefore \(\bm\lambda\in C^r(V).\)

	Finally, \(C^{\UE}(\bm\mu)=\bm\mu^\top\bm\lambda(\bm\mu)\).
	Since \(\bm\mu\mapsto\bm\mu\) is smooth and \(\bm\lambda\in C^r(V)\), we obtain \(C^{\UE}\in C^r(V).\)
	Thus \(\bm\tau(\bm\mu)\in C^r(V)\), \(\bm\lambda(\bm\mu)\in C^r(V)\), and \(C^{\UE}(\bm\mu)\in C^r(V)\).
\Halmos
\end{proof}

\section{Proofs for Section \ref{sec:affine}}
\label{app:sec5-proofs}

\begin{proof}{Proof of \cref{prop:affine-exceptional-polyhedral}}
	Assume that every edge cost is affine: \(c_e(x)=a_ex+b_e\) with \(a_e,b_e\ge0\).
	Let \(A:=\operatorname{diag}(a_e)_{e\in E}\) and \(b:=(b_e)_{e\in E}\).
	Fix an active path set \(\alpha\). After deleting the inactive paths, the Wardrop
	conditions \eqref{eq:wardrop-equal}--\eqref{eq:wardrop-ineq} for the retained paths are
	\[
	C_\alpha^\top(AC_\alpha f_\alpha+b)=R_\alpha^\top\lambda,
	\qquad
	R_\alpha f_\alpha=\bm\mu,
	\]
	together with \(f_\alpha>0\)
	and the inactive-path inequalities
	\(\ell_q^\alpha(f_\alpha)\ge \lambda_i\) for all \(q\in\mathcal P_i\setminus\alpha_i\).
	Because the edge costs are affine, every path cost is affine in \(f_\alpha\). Therefore
	all equalities and inequalities appearing above are affine in \((f_\alpha,\lambda,\bm\mu).\)
	Hence the set of triples \((f_\alpha,\lambda,\bm\mu)\) satisfying the fixed-active-set
	Wardrop system is defined by finitely many affine equalities and (strict as well as
	non-strict) affine inequalities, and is therefore a finite union of polyhedral
	cells. Its projection onto
	the \(\bm\mu\)-coordinates is again a finite union of polyhedral cells, by
	Fourier--Motzkin elimination. Consequently,
	each \(\Omega_\alpha\) is a finite union of polyhedral cells.
	
	Since there are only finitely many active path sets \(\alpha\subseteq\mathcal P\), the
	regular region \(\mathcal O^{\UE}=\bigcup_\alpha \operatorname{int}(\Omega_\alpha)\)
	is a finite union of full-dimensional polyhedral cells. Its complement in \(U^\circ\),
	\(\mathcal E^{\UE}=U^\circ\setminus\mathcal O^{\UE}\),
	is contained in the union of the proper faces of the finitely many polyhedral cells
	appearing in the above finite decomposition: indeed, if
	\(\bm\mu\in\Omega_\alpha\setminus\operatorname{int}(\Omega_\alpha)\) lies in a cell
	\(P\) of the decomposition of \(\Omega_\alpha\), then \(P\) cannot be full-dimensional
	with \(\bm\mu\) in its interior, for otherwise a whole neighborhood of \(\bm\mu\)
	would lie in \(\Omega_\alpha\); hence \(\bm\mu\) lies in a proper face of \(P\) or in
	a lower-dimensional cell. Every proper face has dimension at most
	\(n-1\). Therefore \(\mathcal E^{\UE}\)
	is a finite union of polyhedral cells of dimension at most \(n-1\).
	
	The same argument applies to the transformed network. Indeed, the transformed edge
	costs are \(\bar c_e(x)=c_e(x)+xc'_e(x)=2a_ex+b_e\),
	which are again affine. Thus \(\mathcal E^{\SO}\)
	is also a finite union of polyhedral cells of dimension at most \(n-1\).

	It remains to prove the scaling relation. Let \(\widetilde\Omega_\alpha\) denote the
	active-set region of the transformed network. For any feasible transformed-network
	flow \(f\in\mathcal F(\bm\mu)\), the scaled flow \(2f\) is feasible for the original
	network under demand \(2\bm\mu\). Moreover,
	\[
	c_e\bigl((C(2f))_e\bigr)
	=
	a_e\cdot 2(Cf)_e+b_e
	=
	\bar c_e\bigl((Cf)_e\bigr).
	\]
	Thus every path cost in the original network at the scaled flow \(2f\) and demand
	\(2\bm\mu\) is exactly equal to the corresponding transformed-network path cost at
	\(f\) and demand \(\bm\mu\). Therefore the Wardrop equal-cost and inactive-path
	inequality conditions are identical under the correspondence
	\(f\longleftrightarrow 2f\), \(\bm\mu\longleftrightarrow 2\bm\mu\).
	Hence, for every active path set \(\alpha\),
	\(\bm\mu\in\widetilde\Omega_\alpha\) if and only if \(2\bm\mu\in\Omega_\alpha\).
	Equivalently, \(\widetilde\Omega_\alpha=\frac12\Omega_\alpha\).
	Taking interiors and then finite unions gives
	\(\mathcal O^{\SO}=\frac12\mathcal O^{\UE}\).
	Taking complements in \(U^\circ\) yields
	\(\mathcal E^{\SO}=\frac12\mathcal E^{\UE}\).
	Therefore \(\mathcal E^{\UE}=2\mathcal E^{\SO}:=\{2\bm\mu:\ \bm\mu\in\mathcal E^{\SO}\}\).
	This completes the proof.\Halmos
\end{proof}

\begin{proof}{Proof of \cref{thm:poa-cell-boundary-maximum}}
	Write
	\[
	A:=\operatorname{diag}(a_e)_{e\in\mathcal E},
	\qquad
	\bm b:=(b_e)_{e\in\mathcal E},
	\qquad
	H:=C^\top AC,
	\qquad
	\bm d:=C^\top\bm b,
	\]
	so that \(\Cost(f)=f^\top Hf+\bm d^\top f\) for every path flow \(f\). We first bound the gap \(C^{\UE}-C^{\SO}\) uniformly in the demand, then show that \(\PoA\) attains its maximum on \(\overline K\), and finally show that every interior point of \(K\) is dominated by a boundary point.

	For the first step, fix a demand \(\bm\mu\) and, for \(\tau\in[1/2,1]\), let
	\[
	f_\tau\in\operatorname*{arg\,min}_{f\in\mathcal F(\bm\mu)}
	\Bigl\{\frac12f^\top Hf+\tau\bm d^\top f\Bigr\},
	\qquad
	T_{\bm\mu}(\tau):=\Cost(f_\tau).
	\]
	Then \(f_{1/2}\) is a system-optimal flow and \(f_1\) is a Wardrop equilibrium flow. The induced total cost is independent of the chosen minimizer: two minimizers \(f,g\) for the same \(\tau\) satisfy \(H(f-g)=0\) by convexity, hence \(f^\top Hf=g^\top Hg\), and equality of the objective values with \(\tau>0\) gives \(\bm d^\top f=\bm d^\top g\). The usual subsequence argument on the compact, \(\tau\)-independent set \(\mathcal F(\bm\mu)\) then shows that \(T_{\bm\mu}\) is well defined and continuous.

	For a fixed support \(J\), the restricted KKT conditions read \(H_Jf+\tau\bm d_J=R_J^\top\bm\lambda\) and \(R_Jf=\bm\mu\) with \(f>\bmzero\), together with the KKT inequalities for the paths outside \(J\); here \(H_J,\bm d_J,R_J\) denote the restrictions to \(J\). All these conditions are linear in \((\tau,f,\bm\lambda)\), so, by affine interpolation of solutions, the set of parameters \(\tau\) for which a minimizer with support exactly \(J\) exists is convex, hence an interval. Since there are only finitely many supports, \([1/2,1]\) splits into finitely many subintervals, on the interior of each of which a single support can be used. On such a subinterval, affine interpolation of minimizers and multipliers preserves the KKT conditions, so an affine minimizer can be selected locally; writing its derivative as \(v\) and differentiating the restricted KKT equalities gives \(R_Jv=\bmzero\) and \(H_Jv+\bm d_J=R_J^\top\bm\lambda'\). With \(Z_J\) an orthonormal basis of \(\ker R_J\), \(M_J:=Z_J^\top H_JZ_J\), and \(\bm g_J:=Z_J^\top\bm d_J\), writing \(v=Z_Jw\) yields \(M_Jw=-\bm g_J\); hence
	\[
	\bm d_J^\top v=-\kappa_J,
	\qquad
	\kappa_J:=\bm g_J^\top M_J^\dagger\bm g_J\ge0,
	\]
	where \(M_J^\dagger\) is the Moore--Penrose inverse, with the convention \(\kappa_J=0\) when \(\ker R_J=\{\bmzero\}\). Taking the inner product of the restricted KKT equality with \(v\) gives \(v^\top H_Jf+\tau\bm d_J^\top v=0\), whence, on the subinterval,
	\[
	T_{\bm\mu}'(\tau)
	=
	2v^\top H_Jf+\bm d_J^\top v
	=
	(1-2\tau)\bm d_J^\top v
	=
	(2\tau-1)\kappa_J .
	\]
	The constant \(\kappa_J\) depends only on the network, the cost coefficients, and the support \(J\), not on the demand; let \(\kappa_*\) denote the maximum of \(\kappa_J\) over the finitely many supports. Since \(T_{\bm\mu}\) is continuous and piecewise \(C^1\), piecewise integration gives
	\[
	0\le C^{\UE}(\bm\mu)-C^{\SO}(\bm\mu)
	=
	T_{\bm\mu}(1)-T_{\bm\mu}(1/2)
	\le
	\kappa_*\int_{1/2}^1(2\tau-1)\,d\tau
	=
	\frac{\kappa_*}{4}.
	\]
	Thus there exists a constant \(\Delta<\infty\), independent of the demand, such that
	\begin{equation}
	0\le C^{\UE}(\bm\mu)-C^{\SO}(\bm\mu)\le\Delta .
	\label{eq:uniform-ue-so-gap}
	\end{equation}

	For the second step, delete every OD pair admitting a zero-cost path, i.e., one on which \(a_e=b_e=0\) for all edges: routing the demand of such an OD pair on its zero-cost path changes no edge cost, so, writing \(\pi\) for the projection onto the retained demand coordinates and hats for the reduced problem,
	\[
	C^{\UE}(\bm\mu)=\widehat C^{\UE}(\pi\bm\mu),
	\qquad
	C^{\SO}(\bm\mu)=\widehat C^{\SO}(\pi\bm\mu),
	\]
	where the UE identity uses the value-uniqueness of the equilibrium total cost, and the SO identity uses that \(x\mapsto xc_e(x)\) is nondecreasing. Since \(C^{\SO}>0\) on \(\overline K\), at least one coordinate is retained. Each retained OD pair \(i\) satisfies \(a_p+b_p>0\) on every path \(p\in\mathcal P_i\); writing \(k_i:=|\mathcal P_i|\), \(a_p:=\sum_{e\in p}a_e\), and \(b_p:=\sum_{e\in p}b_e\), at demand \(\nu_i\) some path \(p\in\mathcal P_i\) carries at least \(\nu_i/k_i\), whence
	\(\widehat C^{\SO}(\bm\nu)\ge\min_{p\in\mathcal P_i}\{a_p(\nu_i/k_i)^2+b_p\nu_i/k_i\}\to\infty\) as \(\nu_i\to\infty\),
	and therefore \(\widehat C^{\SO}(\bm\nu)\to\infty\) as \(\|\bm\nu\|\to\infty\). By \eqref{eq:uniform-ue-so-gap},
	\(\widehat\PoA(\bm\nu)\le1+\Delta/\widehat C^{\SO}(\bm\nu)\), so \(\widehat\PoA(\bm\nu)\to1\) as \(\|\bm\nu\|\to\infty\). The function \(\widehat\PoA\) is continuous on the closed polyhedron \(P:=\pi(\overline K)\). If \(\widehat\PoA\equiv1\) on \(P\), its maximum is trivially attained; otherwise, for any \(\bm\nu^0\in P\) with \(\widehat\PoA(\bm\nu^0)>1\), the superlevel set \(\{\bm\nu\in P:\widehat\PoA(\bm\nu)\ge\widehat\PoA(\bm\nu^0)\}\) is bounded by the limit above, and the maximum is attained on \(P\). Since \(P=\pi(\overline K)\), the original \(\PoA\) attains its maximum on \(\overline K\) as well.

	For the last step, fix \(\bm q\in K\) and write \(Q(f):=f^\top Hf\) and \(L(f):=\bm d^\top f\). Since \(\mathcal F(t\bm q)=t\mathcal F(\bm q)\) for \(t>0\),
	\[
	G_{\bm q}(t):=\frac{C^{\SO}(t\bm q)}{t}=\min_{f\in\mathcal F(\bm q)}\{tQ(f)+L(f)\}
	\]
	is concave in \(t\), as the pointwise infimum of affine functions. On \(I_{\bm q}:=\{t>0:t\bm q\in K\}\), the function
	\[
	H_{\bm q}(t):=\frac{C^{\UE}(t\bm q)}{t}=\bm q^\top\bm\lambda(t\bm q)
	\]
	is affine: indeed, interpolating affinely between two Wardrop equilibrium flows with support \(\alpha\) at demands \(t_1\bm q\) and \(t_2\bm q\) preserves the Wardrop conditions, since the path costs are affine in the flow, and the interpolated flow is a Wardrop equilibrium at the interpolated demand \(t_\theta\bm q:=(1-\theta)t_1\bm q+\theta t_2\bm q\); by the value-uniqueness of the equilibrium cost vector, \(\bm\lambda(t_\theta\bm q)=(1-\theta)\bm\lambda(t_1\bm q)+\theta\bm\lambda(t_2\bm q)\). With \(M:=\max\{\PoA(t_1\bm q),\PoA(t_2\bm q)\}\ge0\), the concavity of \(G_{\bm q}\) then gives
	\[
	H_{\bm q}(t_\theta)
	=
	(1-\theta)H_{\bm q}(t_1)+\theta H_{\bm q}(t_2)
	\le
	M\bigl((1-\theta)G_{\bm q}(t_1)+\theta G_{\bm q}(t_2)\bigr)
	\le
	M G_{\bm q}(t_\theta),
	\]
	and dividing by the positive \(G_{\bm q}(t_\theta)\) yields
	\begin{equation}
	\PoA(t_\theta\bm q)\le\max\{\PoA(t_1\bm q),\PoA(t_2\bm q)\}.
	\label{eq:cell-poa-radial-inequality}
	\end{equation}
	Since \(K\) is open and convex, \(I_{\bm q}=(t_-,t_+)\) with \(0\le t_-<1<t_+\le\infty\); moreover \(t_->0\), for \(t_-=0\) would give \(\bmzero\in\overline K\) and \(C^{\SO}(\bmzero)=0\), contradicting the assumption. If \(t_+<\infty\), then \(t_-\bm q,t_+\bm q\in\partial K\), and \eqref{eq:cell-poa-radial-inequality} with \(t_\theta=1\), together with the continuity of \(\PoA\) on \(\overline K\), gives \(\PoA(\bm q)\le\max\{\PoA(t_-\bm q),\PoA(t_+\bm q)\}\). If \(t_+=\infty\), then \(\PoA(t\bm q)\to1\) as \(t\to\infty\) by the second step, and the same limiting procedure gives \(\PoA(\bm q)\le\max\{\PoA(t_-\bm q),1\}=\PoA(t_-\bm q)\). In either case, \(\bm q\) is dominated by a point of \(\partial K\); applied to a maximizer of \(\PoA\) on \(\overline K\), which exists by the second step, this yields a maximizer on \(\partial K\).\Halmos
\end{proof}

\bibliographystyle{informs2014}
\bibliography{references}

\end{document}